\documentclass[12pt]{amsart}
\usepackage{tgpagella}
\usepackage{euler}
\usepackage[T1]{fontenc}
\usepackage{amsmath, amssymb}
\usepackage[hidelinks]{hyperref}
\usepackage[english]{babel}
\usepackage{mathrsfs}
\usepackage{eucal}
\usepackage[all]{xy}

\newtheorem{thm}{Theorem}[subsection]
\newtheorem*{thm*}{Theorem}
\newtheorem{lem}[thm]{Lemma}
\newtheorem{fact}[thm]{Fact}
\newtheorem{prop}[thm]{Proposition}
\newtheorem*{prop*}{Proposition}

\newtheorem{cor}[thm]{Corollary}

\theoremstyle{definition}
\newtheorem{defn}[thm]{Definition}
\newtheorem{notation}[thm]{Notation}
\newtheorem{remark}[thm]{Remark}
\newtheorem{question}[thm]{Question}
\newtheorem{example}[thm]{Example}

\newcommand{\dminus}{ 
\buildrel\textstyle\ .\over{\hbox{ 
\vrule height3pt depth0pt width0pt}{\smash-} 
}}

\def\ca{\curvearrowright}

\def\la{\lambda}

\def\om{\omega}

\def\vp{\varphi}

\def\fr{\mathfrak}

\def\sg{\sigma}

\def\cal{\mathcal}

\def\M{\mathcal{M}}
\newcommand{\cstar}{$\mathrm{C}^*$}
\def\bbR{\mathbf R}
\def \svna{\sigma\text{-}\operatorname{vNa}}

\newcommand\ab[1]{\left| #1 \right|}
\newcommand\ip[2]{\left\langle #1\, ,\!\ #2 \right\rangle}
\newcommand\cU{{\cal U}}
\newcommand\cL{{\cal L}}

\newcommand\bbN{{\mathbf N}}
\def\corr{\operatorname{corr}}
\def \Th{\operatorname{Th}}

\begin{document}


\title{Correspondences, ultraproducts, and model theory}

\author{I. Goldbring, B. Hart and T. Sinclair}
\thanks{I. Goldbring was partially supported by NSF CAREER grant DMS-1349399.}
\thanks{T. Sinclair was partially supported by NSF grant DMS-1600857.}
\address{Department of Mathematics\\University of California, Irvine, 340 Rowland Hall (Bldg.\# 400),
Irvine, CA 92697-3875}
\email{isaac@math.uci.edu}
\urladdr{http://www.math.uci.edu/~isaac}

\address{Department of Mathematics and Statistics, McMaster University, 1280 Main St., Hamilton ON, Canada L8S 4K1}
\email{hartb@mcmaster.ca}
\urladdr{http://ms.mcmaster.ca/~bradd/}

\address{Mathematics Department, Purdue University, 150 N. University Street, West Lafayette, IN 47907-2067}
\email{tsincla@purdue.edu}
\urladdr{http://www.math.purdue.edu/~tsincla/}


\begin{abstract}
We study correspondences of tracial von Neumann algebras from the model-theoretic point of view.  We introduce and study an ultraproduct of correspondences and use this ultraproduct to prove, for a fixed pair of tracial von Neumann algebras $M$ and $N$, that the class of $M$-$N$ correspondences forms an elementary class.  We prove that the corresponding theory is classifiable, all of its completions are stable, that these completions have quantifier elimination in an appropriate language, and that one of these completions is in fact the model companion.  We also show that the class of triples $(M,H,N)$, where $M$ and $N$ are tracial von Neumann algebras and $H$ is an $M$-$N$ correspondence, form an elementary class.  As an application of our framework, we show that a II$_1$ factor $M$ has property (T) precisely when the set of central vectors form a definable set relative to the theory of $M$-$M$ correspondences.  We then use our approach to give a simpler proof that the class of structures $(M,\varphi)$, where $M$ is a $\sigma$-finite von Neumann algebra and $\varphi$ is a faithful normal state, forms an elementary class.  Finally, we initiate the study of a family of Connes-type ultraproducts on \cstar-algebras.
\end{abstract}

\maketitle

\tableofcontents

\section{Introduction} 

The papers \cite{FHS2, FHS3} provide a model theoretic treatment of the theory of II$_1$ factors. A striking feature of the axiomatization of II$_1$ factors therein is that the sorts (domains of quantification) for a II$_1$ factor $M$ are bounded balls in the operator norm, while the metric is derived from the weaker $2$-norm which metrizes the strong operator topology. Thus, while each sort is complete, the entire structure $M$ is not. A key observation of the second author was that the completion of $M$ in its $2$-norm, denoted $L^2(M)$, does however belong to the space of imaginaries of the theory described in \cite{FHS2}. The Hilbert space $L^2(M)$ on which $M$ naturally acts on the left and right is an extremely important object associated to the II$_1$ factor $M$, known as its standard form. Every von Neumann algebra admits a canonical standard form which is so intertwined with the theory of the algebra that a working von Neumann algebraist would find it nearly impossible to practically divorce the two.

The motivations for the paper arose from several different, but interconnected, trains of the thoughts the authors had begun to pursue. One such thought was to develop a model theoretic understanding of Connes' derivation of the ultraproduct of II$_1$ factors through points of continuity in the Hilbert space ultraproduct of the standard forms \cite{connes76}. This was very much related to a second goal, namely that of providing a model theoretic treatment of a fuller range of properties of II$_1$ factors, including the II$_1$ factor versions of amenability and Kazhdan's property (T), as well as their relativizations. Unlike the McDuff property and property Gamma, both of which are axiomatizable, these properties are not defined in  terms of the algebra, but rather in terms of the category whose morphisms are the trace-preserving, unital completely positive maps. Hilbert spaces with commuting normal representations, referred to as \emph{Hilbert bimodules} or \emph{correspondences}, give a general framework for understanding these morphisms. 

Correspondences for pairs of II$_1$ factors were introduced by Connes in the early 1980s as a way of importing ideas from representation theory and ergodic theory into the setting of finite von Neumann algebras. The fundamentals of the theory were subsequently developed in \cite{cojo}, where the notion of property (T) for II$_1$ factors was defined and investigated, and in \cite{Popa86}. The study of correspondences is of central importance to the theory of II$_1$ factors, providing the conceptual framework behind many of the breakthrough techniques in Popa's deformation/rigidity theory. (See, for instance, \cite{PopaBetti}.)

To a model theorist, the definition of property (T) is a definability statement.  Indeed, in \cite{Gold}, the first-named author observed that a countable discrete group $\Gamma$ has property (T) if and only if the set of fixed vectors is a definable set relative to the theory of unitary representations of $\Gamma$.  In order to prove an analogous result in the case of II$_1$ factors, one is confronted with the issue that the class of correspondences is not on the face of it an elementary class. We correct this here. Specifically, for any pair of tracial von Neumann algebras $M$ and $N$, we prove that the class of $M$-$N$ correspondences forms an elementary class. We then prove that for a  II$_1$ factor $M$, the set of central subtracial vectors is a definable set relative to the theory of $M$-$M$ correspondences.  

A third motivation came from the work of Dabrowski \cite{Dab}, who gave an axiomatization for the class of all $\sg$-finite von Neumann algebras, and the related work of Ando and Haagerup \cite{AH} on ultraproducts of type III factors. In attempting to parse the model theoretic content of the latter paper, it was realized by the second and third authors that an explicit choice of a model of the standard form of $M$ given by the GNS construction associated to a faithful normal state was crucial to a model-theoretic understanding of the Ocneanu ultraproduct which could then potentially be used to provide an alternate, simpler treatment of the former paper. Lastly, there was the goal of using model theory to provide a formal structure for the transfer between certain classification result in C$^*$-algebras and their von Neumann algebraic counterparts. 

In Section 2, we gather the necessary background on correspondences needed for the main results in Sections 3 and 4; in particular, we develop the relevant theory of \emph{bounded} vectors in correspondences, a notion crucial to the model-theoretic presentation of correspondences.  In Section 3, we introduce an ultraproduct construction for correspondences and prove some technical results about this construction that are used in the subsequent section.

Section 4 contains the model-theoretic core of the paper.  We show that the class of $M$-$N$ correspondences is an elementary class and use this to give a model-theoretic characterization of the notion of weak containment of correspondences.  We also show that the theory of correspondences is classifiable, stable, and admits a model companion.  In the final subsection of Section 4, we mention that the class of triples $(M,H,N)$, where $M$ and $N$ are tracial von Neumann algebras and $H$ is an $M$-$N$ correspondence, is also an elementary class.

In Section 5, we use our framework to give a model-theoretic characterization of property (T) for II$_1$ factors analogous to that proven for countable discrete groups by the first-named author in \cite{Gold}; to wit:  a separable II$_1$ factor $M$ has property (T) precisely when the set of central vectors is a definable set relative to the theory of $M$-$M$ correspondences.

In Section 6, we use the perspective of bounded vectors to give an alternative (and more elementary) proof that the class of pairs $(M,\varphi)$, where $M$ is an arbitrary (not necessarily finite) von Neumann algebra and $\varphi$ is a faithful normal state on $M$, is an elementary class, a result first proved by Dabrowski in \cite{Dab}.  We show that the resulting model-theoretic ultraproduct coincides with the classical ultraproduct of such pairs constructed by Ocneanu.  We use this latter fact and results of Ando-Haagerup \cite{AH} to show that the modular group of $(M,\varphi)$ is definable and to determine whether or not some natural classes of pairs $(M,\varphi)$ are axiomatizable or local.

In Section 7, we sketch the beginnings of an approach to studying certain completions of C$^*$-algebras and their associated ultraproducts which in a sense lie between the theory of C$^*$-algebras and that of von Neumann algebras. This work is inspired by the study of W$^*$-bundles initiated by Ozawa \cite{Oz}. Our approach here is again a more spatial one, defining bi-modular structures of which the operator algebras arise as the bounded objects. There is clearly much work left to be done here which will be undertaken in future papers.

We assume that the reader is familiar with the basic theory of von Neumann algebras as well as the continuous model theory used to study them.  A reader unacquainted with the latter can refer to \cite{FHS2} or \cite{munster}.  \cite{Gold} is a good reference for definability in continuous model theory. 

\subsection*{Acknowledgements} The authors are indebted to Hiroshi Ando for discussions on the Ocneanu ultraproduct and the main results in \cite{AH}.  We would also like to thank Vern Paulsen who encouraged the second author to look at the model theory of correspondences and both Jesse Peterson and David Sherman for interesting initial conversations. We thank Chris Schafhauser and Gabor Szabo for useful remarks on an earlier version of this manuscript.
The first author would like to thank the Institut Henri Poincar\' e for their hospitality during the trimester on model theory, valued fields, and combinatorics where some of this work took place. Finally, we would like to thank the anonymous referee for pointing out several errors in the manuscript, and for making many suggestions which improved the exposition.

\section{Correspondences} In this section we develop the necessary operator algebraic background which will be used in the proofs of the main results.

\subsection{Reminders on states and traces}  Since states, traces and positive maps play a critical role throughout this paper, we remind the reader of the definitions.
\begin{defn}
Suppose that $A$ and $B$ are $*$-algebras.
\begin{enumerate}
\item We say that an element of $A$ is \textbf{positive} if it is of the form $a^*a$ for some $a \in A$.
\item A linear map $\phi:A\to B$ is called \textbf{positive} if it sends positive elements to positive elements.
\item A \textbf{state} on $A$ is a positive linear functional on $\phi:A\to \mathbb C$ (that is, one that sends positive elements of $A$ to positive real numbers) such that $\phi(1) = 1$.
\item A state is \textbf{faithful} if the only positive element sent to 0 is 0.
\item A state on $A$ is \textbf{normal} if it is continuous with the weak operator topology when restricted to the unit ball of $A$.
\item A \textbf{trace} on $A$ is a state $\tau$ which satisfies the tracial property $\tau(ab) = \tau(ba)$ for all $a, b \in A$.
\end{enumerate}
\end{defn}

One of the key constructions that is carried out on a *-algebra $A$ with a state $\phi$ is the \textbf{GNS construction}: for $a, b \in A$, define the sesquilinear form
\[
\langle a, b \rangle = \phi(b^*a).
\]
Completing $A$ with respect to this form, we get a Hilbert space on which $A$ acts and we denote this space by $L^2(A,\phi)$.  

When we say that a pair $(M,\tau)$ is a tracial von Neumann algebra this is understood to mean that $M$ is a von Neumann algebra and $\tau$ is a fixed, faithful, normal trace on $M$.

\subsection{Introduction to Correspondences}

In this section, we include some background on correspondences.  The treatment follows chapter 13 in \cite{Popa-book}. Fix two tracial von Neumann algebras $(M,\tau_M)$ and $(N,\tau_N)$.

We recall that a normal representation of $M$ on the Hilbert space $H$ is a $*$-homomorphism $\pi:M\to \mathcal B(H)$ such that the restriction of $\pi$ to the unit ball of $M$ is continuous in the weak operator topology.  We also recall that $M^{\operatorname{op}}$ is the von Neumann algebra obtained from $M$ by redefining multiplication so that $a\cdot^{\operatorname{op}}b:=ba$.

\begin{defn}
An $M$-$N$ \textbf{correspondence} (or an $M$-$N$ \textbf{bimodule}) is a Hilbert space $H$ together with normal commuting representations $\pi_M$ and $\pi_N$ of $M$ and $N^{\operatorname{op}}$ respectively on $H$, that is, for $a\in M$, $b\in N$, and $\xi\in H$, we have 
$$\pi_M(a)(\pi_N(b)(\xi))=\pi_N(b)(\pi_M(a)(\xi)).$$  We typically refer to the action of $M$ on $H$ as the left action and write $m\xi$ instead of $\pi_M(m)\xi$.  Similarly we refer to the action of $N^{\operatorname{op}}$ as the right action and will write $\xi n$ for $\pi_N(n)\xi$.
\end{defn}

\begin{example}

\

\begin{enumerate}
\item $L^2(M,\tau_M)$ is an $M$-$M$ correspondence, called the \textbf{trivial $M$-$M$ correspondence}.  Here, the actions of $M$ and $M^{\operatorname{op}}$ are given by the extensions of the left and right actions of $M$ on itself to $L^2(M,\tau_M)$.
\item $L^2(M,\tau_M)\otimes L^2(N,\tau_N)$ (the tensor product of the two Hilbert spaces)  is an $M$-$N$ correspondence, called the \textbf{coarse $M$-$N$ correspondence}.  The left and right actions are given by 
\[
\pi_M(a)(x\otimes y):=(ax\otimes y) \mbox{ and } \pi_N(b)(x\otimes y):=x\otimes (yb).
\]
\end{enumerate}
\end{example}

\begin{example}
We show here the manner in which correspondences capture unitary representations.  This summarizes subsection 13.1.3 of \cite{Popa-book}.

Fix a discrete group $\Gamma$. Recall that $\ell^2(\Gamma)$ is the Hilbert space generated by an orthonormal basis $\{\delta_\eta : \eta \in \Gamma\}$. For each $\gamma \in \Gamma$ we associate $u_\gamma$, a unitary operator on $\ell^2(\Gamma)$ defined by
\[
u_\gamma(\delta_\eta) = \delta_{\eta\delta}.
\]
$L(\Gamma)$ is the von Neumann algebra generated by $\{u_\gamma : \gamma \in \Gamma\}$; we will denote $L(\Gamma)$ by $M$.  Note that this defines a unitary representation called the left regular representation, $\pi_\ell$, of $\Gamma$ on the Hilbert space $\ell^2(\Gamma)$.  Every element of $M$ can be formally written as $\displaystyle{\sum_{\gamma \in \Gamma} c_\gamma u_\gamma}$ for complex numbers $c_\gamma$.  If we define
$\displaystyle{\tau(\sum_{\gamma \in \Gamma} c_\gamma u_\gamma) = c_e}$ where $e$ is the identity in $\Gamma$ then $(M,\tau)$ is a tracial von Neumann algebra.  One checks that in fact, $\ell^2(\Gamma) = L^2(M,\tau)$.

Now suppose that $\pi:\Gamma\to U(K)$ is a unitary representation of a discrete group $\Gamma$ on some Hilbert space $K$.   We give $K\otimes \ell^2(\Gamma)$ the structure of a $M$-$M$ correspondence by setting, for $\xi\in K$ and $\gamma,\eta\in \Gamma$, 
\[
u_\gamma (\xi\otimes \delta_\eta):=\pi(\gamma)(\xi)\otimes \delta_{\gamma\eta} \mbox{ and }(\xi\otimes \delta_{\eta})u_{\gamma}:=\xi\otimes \delta_{\eta\gamma}.
\]
 One show that these definitions extend from $\Gamma$ to all of $M$.  
 
 Two particular instances of this construction are the trivial and left regular representations.  That is, if $\pi:\Gamma \rightarrow \mathbf{C}$ is defined by $\pi(\gamma) = 1$ i.e. the trivial representation, then the above construction yields the correspondence $\ell^2(\Gamma)$ which we noted is just $L^2(M,\tau)$, the trivial correspondence. If we apply the construction to the left regular representation $\pi_\ell$ then the resulting correspondence is the coarse correspondence on $L^2(M,\tau) \otimes L^2(M,\tau)$.
\end{example}

\subsection{Bounded vectors}

The following definition encompasses one of the key notions in this paper:

\begin{defn} If $H$ is an $M$-$N$ correspondence and $\xi \in H$, then we say that $\xi$ is \textbf{left ($K$)-bounded} if:  for all $c \in M_+$ (positive elements in $M$), we have
\[
\langle c\xi,\xi \rangle \leq K \tau_M(c).
\]
Similarly, we say that $\xi$ is \textbf{right ($K$)-bounded} if: for all $c \in N_+$, we have 
\[
\langle \xi c,\xi \rangle \leq K \tau_N(c).
\]
$\xi$ is called \textbf{bounded} if it is both left and right bounded.  If we wish to highlight specific left and right bounds, we will say $\xi$ is \textbf{$(K,L)$-bounded} if it is left $K$-bounded and right $L$-bounded. We say that $\xi$ is \textbf{$K$-bounded} if it is $(K,K)$-bounded.  $1$-bounded vectors are often called \textbf{subtracial} in the literature.
\end{defn}

\begin{notation} For $\xi \in H$, let $L_\xi\colon M \rightarrow H$ and $R_\xi\colon N \rightarrow H$ be defined by $L_\xi(a):=a\xi$ and $R_\xi(b):=\xi b$.
\end{notation}

The following lemma explains the terminology:

\begin{lem}
If $H$ is an $M$-$N$ correspondence and $\xi \in H$, then the following are equivalent:
\begin{enumerate}
\item $\xi$ is left $K$-bounded.
\item $L_\xi$ extends to a bounded linear map $L_\xi:L^2(M,\tau_M)\to H$ with $\|L_\xi\|\leq \sqrt{K}$.
\end{enumerate}
An analogous statement holds for right $K$-bounded vectors.
\end{lem}

\proof Suppose that $c \in M$ and $c \neq 0$.  Then
\[
\langle c\xi, c\xi \rangle = \langle c^*c\xi,\xi \rangle \leq K\tau_M(c^*c).
\]
So $\| L_\xi\| \leq \sqrt{K}$.  The reverse direction is similar.
\endproof

\begin{example}\label{boundedstandard}
In the trivial $M$-$M$ correspondence $L^2(M,\tau)$, an element $b\in M$ is left (resp.\ right) $K$-bounded if $bb^*\leq K\cdot 1$ (resp.\ $b^*b\leq K\cdot 1$) or equivalently $\|b\| \leq \sqrt{K}$.  To see this, we use a small fact which recurs throughout the paper: if $(M,\tau)$ is a tracial von Neumann algebra and $b,c \in M$ then $\tau(bc) \leq \|b\|\tau(c)$.  So suppose that $b, c \in M$ and $bb^* \leq K \cdot 1$.  Then
\[
\langle cb,cb \rangle_\tau = \tau(b^*c^*cb) = \tau(bb^*c^*c) \leq \|bb^*\|\tau(c^*c) \leq K\tau(c^*c).
\]
If $\|b\| \leq \sqrt{K}$ then by the \cstar-identity, $\|bb^*\| = \|b\|^2$ and we get the same conclusion.
\end{example}

We will need the following lemma in the next section.

\begin{lem} \label{basic}Suppose $H$ is an $M$-$N$ correspondence.  Then the set of bounded vectors forms a linear subspace of $H$ closed under the left and right actions of $M$ and $N$.  Moreover:
\begin{enumerate}
\item Any convex combination of left (resp.\ right) $K$-bounded vectors is left (resp.\ right) $K$-bounded.
\item If $\xi_1,\xi_2 \in H$ are $K_1$ and $K_2$ left (resp.\ right) bounded respectively, then $\xi_1 + \xi_2$ is left
\[
\left (\sqrt{K_1} + \sqrt{K_2}\right )^2
\]
 (resp.\ right) bounded.
\item If $\xi\in H$ is $(K,L)$-bounded, then for every $a\in M$ and $b\in N$, we have that $a\xi$ is $(\|a\|^2K,\|a\|^2L)$-bounded and $\xi b$ is $(\|b\|^2K,\|b\|^2L)$-bounded.
\end{enumerate}
\end{lem}

\proof For (1), suppose that $\xi = \sum_i \lambda_i\zeta_i$ is a convex sum of left $K$-bounded vectors, so $\|L_{\zeta_i}\| \leq \sqrt{K}$ for all $i$. Then
\[
\|L_\xi\| \leq \sum_i\lambda_i\|L_{\zeta_i}\| = \sqrt{K}
\]
and so $\xi$ is left $K$-bounded.  A similar calculation works on the right.

For (2), we compute $\|L_{\xi_1 + \xi_2}\|$:
\[
\| L_{\xi_1 + \xi_2}\| \leq \|L_{\xi_1}\| + \|L_{\xi_2}\| \leq \sqrt{K_1} + \sqrt{K_2}.
\]
So $\xi_1 + \xi_2$ is left $\left (\sqrt{K_1} + \sqrt{K_2}\right )^2$-bounded.  The calculation is the same on the right.

For (3), we only do the right bounded calculation for $\xi b$ as a good exercise with the right representation.  Suppose that $b \in N$ and $c \in N_+$.  We have
\begin{multline*}
\langle \xi bc,\xi  b\rangle = \langle \pi_r(bc)\xi,\pi_r(b)\xi \rangle = \langle \pi_r(b)^*\pi_r(bc)\xi,\xi \rangle = \langle \pi_r(bcb^*)\xi,\xi \rangle\\
\leq L\tau_N(bcb^*) = L \tau_N(b^*bc) \leq L\|b^*b\|\tau_N(c) = \|b\|^2L\tau_N(c).
\end{multline*}
So $\xi b$ is right $\|b\|^2L$-bounded. \qed

Before the next Lemma we remind the reader of the following definition.
\begin{defn} If $M$ is a von Neumann algebra and $c \in M_+$ then the {\bf support projection} for $c$ is the smallest projection $\pi$ such that $c \leq \pi c$.  As $M$ is a von Neumann algebra, the support projection always exists.
\end{defn}

\begin{lem}\label{cutoff} If $H$ is an $M$-$N$ correspondence and $\xi \in H$, then the following are equivalent:
\begin{enumerate}
\item $\xi$ is left $K$-bounded.
\item There is $b_\xi \in M_+$ such that $\|b_\xi\| \leq K$ and, for all $c \in M_+$, we have
\[
\langle c \xi,\xi \rangle = \tau_M(cb_\xi).
\]
\end{enumerate}

If $\xi$ is left $K$-bounded, then, using the above notation, we also have:
%

\begin{enumerate}
\item[(a)] $Mb_{\xi}^{1/2}$ with the inner product given by $\tau_M$ is isometric to $M\xi$ via the map sending $mb_{\xi}^{1/2}$ to $m\xi$ for all $m \in M$.
\item[(b)] For $R>0$, set $f_R(t):=\min\{t,R\}$ and set $b_{\xi}^{(R)}:=f_R(b_\xi)$.  Let $\pi^{(R)}$ denote the support projection of $b_\xi-b_\xi^{(R)}$.  Then $(1-\pi^{(R)})\xi$ is left $R$-bounded.
\end{enumerate}
An analogous statement holds for right $K$-bounded vectors.
\end{lem}

\begin{proof}
The equivalence of (1) and (2) is a consequence of the so-called Little Radon-Nikodym theorem (see, for example, \cite[Proposition 7.3.6]{Popa-book}).  The proof of (a) is a straightforward computation.  We now prove (b).
For $c\in M_+$, we have
\[
\langle c(1-\pi^{(R)})\xi,(1-\pi^{(R)})\xi\rangle=\tau((1-\pi^{(R)})c(1-\pi^{(R)})b)\leq\tau(cb^{(R)})\leq R\tau(c),
\]
where the last inequality uses the fact that $b^{(R)}\leq R\cdot 1$.
\end{proof}

Bounded vectors are ubiquitous as the following result shows.

\begin{fact}[\cite{Popa-book}, Lemma 13.1.11]
If $H$ is an $M$-$N$ correspondence and $\xi \in H$, then there is a subtracial $\xi_0 \in M\xi N$ such that $\xi \in \overline{M\xi_0 N}$ (the closed Hilbert subspace generated by the vectors $m\xi n$ for all $m \in M$ and $n \in N$). Consequently, the bounded vectors are dense in any correspondence.
\end{fact}

We call a correspondence $H$ \textbf{cyclic} if $H = \overline{M\xi N}$ for some $\xi \in H$.  For a correspondence $H$, a collection $S$ of subcorrespondences of $H$ is \textbf{independent} if: for every $H_0 \in S$, 
\[
H_0 \cap \oplus_{H' \in S \setminus \{H_0\}} H' = \{0\}.
\]
The following Corollary is discussed at the end of section 1 in \cite{Popa86}; it follows immediately from the previous fact and Zorn's lemma.

\begin{cor}
Every $M$-$N$ correspondence $H$ is the closure of the direct sum of any maximal independent collection of cyclic subcorrespondences. 
\end{cor}

\subsection{Correspondences and completely positive maps}

For purely set theoretic reasons, we know there are only a sets worth (boundedly) many cyclic correspondences on cardinality grounds alone.  In fact, we have a more concrete description of these correspondences.  The following is a summary of the discussion in section 13.1.2 of \cite{Popa-book}. 

We say that $\phi\colon M \rightarrow N$ is a completely positive (c.p.) map. if for every $n$, the map
$\phi^{(n)}\colon M \otimes M_n(\mathbf{C} )\rightarrow N \otimes M_n(\mathbf{C})$ given by
\[
\phi^{(n)}( (m_{ij} )_{ij} ) = ( \phi(m_{ij}) )_{ij}
\] is a positive map; $\phi^{(n)}$ is the map which takes an $n \times n$ matrix of entries from $M$ and sends them to their image under $\phi$.

Now given a c.p. map $\phi\colon M \rightarrow N$, we define a correspondence $H_\phi$ as follows: on $M \otimes N$, define the sesquilinear form $\langle \cdot,\cdot \rangle_\phi$ such that
\[
\langle m_1 \otimes n_1,m_2 \otimes n_2 \rangle_\phi = \tau_N(\phi(m_2^*m_1)n_2^*n_1).
\]
Set $H_\phi$ to be the completion of $M \otimes N$ with respect to $\langle \cdot,\cdot \rangle_\phi$.  Similarly, for a c.p.\ map $\psi\colon N \rightarrow M$, one can define a correspondence $_\psi H$ as the completion of the inner product induced by
\[
\ip{m_1 \otimes n_1}{m_2 \otimes n_2} = \tau_M(\psi(n_2^*n_1)m_2^*m_1)
\]
on $M \otimes N$.

Conversely, suppose that $H$ is a correspondence and $\xi \in H$ is right bounded.  We define $\phi_\xi\colon M \rightarrow N$ by
\[
\phi_\xi(m) = R_\xi^* m R_\xi.
\]
$\phi_\xi$ is a c.p. map but more importantly, the codomain of $\phi_\xi$ is $N$.  A priori, for any $m \in M$, $\phi_\xi(m) \in B(L^2(N,\tau_N))$. One uses critically  that $N$ is a von Neumann algebra and that the actions commute to conclude that in fact, $\phi_\xi(m) \in N$.
\begin{fact}
If $\xi$ is a right bounded vector in a correspondence $H$, then $H_{\phi_\xi}$ is isomorphic to $\overline{M\xi N}$ via the map which sends $1 \otimes 1$ to $\xi$.
\end{fact}

\begin{cor} \label{dirsum}
Every correspondence is the direct sum of cyclic correspondences of the form $H_\phi$ where $\phi$ is a c.p.\ map associated to a subtracial vector.
\end{cor}

\section{The correspondence ultraproduct}\label{corrultra}  
\subsection{The case of fixed algebras}  In this section, we continue to fix a pair of tracial von Neumann algebras $(M,\tau_M)$ and $(N,\tau_N)$.  Suppose that $(H_i \ : \ i\in I)$ is a family of $M$-$N$ correspondences and that $\cU$ is an ultrafilter on $I$.  Let $\prod_\cU H_i$ denote the ordinary Hilbert space ultraproduct.  For each $a\in M$, $b\in N$, and $\xi=(\xi_i)^\bullet\in \prod_\cU H_i$, it makes sense to define $a\cdot \xi:=(a\xi_i)^\bullet$ and $\xi\cdot b:=(\xi_ib)^\bullet$.  (In other words, if $\xi=(\xi_i')^\bullet$ for some other sequence $(\xi_i')$, then it is readily verified that $(a\xi_i)^\bullet=(a\xi_i')^\bullet$ and $(\xi_ib)^\bullet=(\xi_i'b)^\bullet$.)  Of course, the resulting actions of $M$ and $N^{\operatorname{op}}$ need not be normal and thus, in general, $\prod_\cU H_i$ is not an $M$-$N$ correspondence.  

That being said, given $\xi\in \prod_\cU H$, we say that the above actions are \textbf{continuous at $\xi$} if left  and right multiplication $L_\xi\colon M \rightarrow \prod_\cU H_i$ and $R_\xi\colon N \rightarrow \prod_\cU H_i$ can be extended to bounded linear maps on $L^2(M,\tau_M)$ and $L^2(N,\tau_N)$ respectively.

\begin{defn}
Suppose that $(H_i \ : \ i\in I)$ is a family of $M$-$N$ correspondences and that $\cU$ is an ultrafilter on $I$.  The \textbf{correspondence ultraproduct} of the family, denoted $\prod_{c\cU} H_i$, is the closure of the set of all $\xi \in \prod_{\cU}H_i$ at which the actions are continuous.   
\end{defn}

Note that $\prod_{c\cU}H_i$ is a closed subspace of $\prod_{\cU}H_i$ and that the induced actions of $M$ and $N$ on $\prod_{c\cU}H_i$ are normal, whence  $\prod_{c\cU} H_i$ is in fact a correspondence.

The goal of the next section is to understand the correspondence ultraproduct as a model-theoretic ultraproduct, which will in turn allow us to view the class of correspondences as an elementary class.  The purpose of this subsection is to prove the necessary technical results to facilitate this development.

Note that a $K$-bounded element of $\prod_\cU H_i$ is necessarily an element of $\prod_{c\cU} H_i$.  We now define an a priori stronger notion.

\begin{defn}
$\xi\in \prod_\cU H_i$ is \textbf{uniformly $K$-bounded} if there are $\xi_i\in H_i$ which are $K$-bounded and for which $\xi=(\xi_i)^\bullet$.  We say that $\xi$ \textbf{uniformly bounded} if $\xi$ is uniformly $K$-bounded for some $K$.
\end{defn}

The following lemma is clear:

\begin{lem}

\

\begin{enumerate}
\item Every uniformly $K$-bounded vector is $K$-bounded.
\item The set of uniformly $K$-bounded vectors (for a fixed $K$) forms a closed set.
\item The set of uniformly bounded vectors forms a subspace of $\prod_{c\cU} H_i$ invariant under the actions of $M$ and $N$.  
\end{enumerate}
\end{lem}

The main goal of this subsection is to prove that every $K$-bounded vector is actually uniformly $K$-bounded (Proposition \ref{mainpropultra} below).  We first prove this result using an extra hypothesis.

\begin{lem}\label{KDTultra}
Suppose that $\xi$ is a $K$-bounded, uniformly bounded element of $\prod_\cU H_i$.  Then $\xi$ is uniformly $K$-bounded.
\end{lem}

\proof This proof is similar to the proof of Lemma 13.1.11 in \cite{Popa-book}.  We can assume for the proof that $K = 1$. Since $\xi$ is uniformly bounded, there is $L\in \mathbb N$ such that we can write $\xi=(\xi_i')^\bullet$ with each $\xi'_i\in H_i$ $L$-bounded.  By Lemma \ref{cutoff}, for each $i \in I$, we have $b'_i \in M_+$, $d'_i \in N_+$ with $\|b'_i\|, \|d'_i\| \leq L$ such that: for all $a \in M_+$, $c \in N_+$, we have
\[
\ip{a\xi'_i}{\xi'_i} = \tau_M(ab'_i) \text{ and } \ip{\xi'_i c}{\xi'_i} = \tau_N(d'_i c).
\]
Similarly, we have $b\in M, d \in N$ such that $\|b\|,\|d\| \leq 1$ and for all $a \in M_+$, $c \in N_+$, we have
\[
\ip{a\xi}{\xi} = \tau_M(ab) \text{ and } \ip{\xi c}{\xi} = \tau_N(dc).
\]
For $0\leq t\leq L$, set $f(t) = \min\{1,t^{-1/2}\}$. Set $b_i := f(b'_i)$ and $d_i := f(d'_i)$.  Notice that $\|b_i\|, \|d_i\| \leq 1$.  Finally, set $\xi_i := b_i\xi'_i d_i$.  

We first show that each $\xi_i$ is 1-bounded.  Towards this end, note that, for $a \in M_+$, we have
\[
\ip{a\xi_i}{\xi_i} \leq \ip{ab_i\xi'_i}{b_i\xi'_i} = \tau_M(b_iab_ib'_i) \leq \tau_M(a)\|b^2_ib'_i\| \leq \tau_M(a),
\]
as, by functional calculus, $b_i^2b'_i$ has norm at most 1.  The calculation on the right is similar.

It remains to see that $\xi=(\xi_i)^\bullet$.  Now ostensibly $(b'_i)^\bullet$ is an element of $M^\cU$ as $\xi$ is uniformly bounded.  However, as $\xi=(\xi_i')^\bullet$, we have, for all $c \in M_+$, that $\lim_\cU \tau_M(cb'_i) =\tau_M(cb)$.  It follows that $b = (b'_i)^\bullet$, that is, that $\lim_\cU b_i=b$.
Since $\|b\| \leq 1$, $f(b) = 1$.  So $1 = f(b) = (f(b'_i))^\bullet = (b_i)^\bullet$ and $\lim_\cU \tau_M(b_i )=1$.  The same is true on the right and so we have $\lim_\cU \tau_N(d_i) = 1$. Now notice that if $\eta$ is $L$-bounded and $a \in M_+$ then 
\[
\| a\eta \|^2 = \langle a^*a\eta,\eta \rangle \leq \tau_M(a) \|a\| L.
\]
A similar inequality is true on the right. It follows that, for any $i\in I$, we have
\[
\begin{split}
\|\xi'_i - \xi_i \| \leq \|\xi'_i - b_i\xi'_i\| + \|b_i\xi'_i - \xi_i\| & \leq  \|(1-b_i)\xi'_i\| + \|b_i\xi'_i(1-d_i)\| \\
										& \leq  \|(1-b_i)\xi'_i\| + \|\xi'_i(1-d_i)\|.
\end{split}
\]
Since the ultralimit of the last quantity tends to 0, we conclude that $\xi=(\xi_i)^\bullet$.   \qed

It is worth singling out a special case of Lemma \ref{KDTultra} (and its proof).
\begin{lem} \label{KDT} Assume that $H$ is an $M$-$N$ correspondence and $H_0$ is a subspace of $H$ with $MH_0N \subseteq H_0$.  Further suppose that $\xi'_i \rightarrow \xi$, where $\xi'_i \in H_0$ is $L$-bounded for each $i \in \mathbf{N}$ and $\xi \in H$ is $K$-bounded.  Then there are $K$-bounded $\xi_i \in H_0$ such that $\xi_i \rightarrow \xi$.
\end{lem}
%
%
%



The following two Propositions are the main results of this subsection.
\begin{prop} \label{mainprop} Assume that $H$ is an $M$-$N$ correspondence and that $H_0$ is a subspace of $H$ with $MH_0N \subseteq H_0$.  Further suppose that $\xi_i \rightarrow \xi$, where $\xi_i \in H_0$ is $K_i$-bounded for each $i \in \mathbf{N}$ and $\xi \in H$ is $K$-bounded.  Then there are $K$-bounded $\eta_i \in H_0$ such that $\eta_i \rightarrow \xi$.
\end{prop}

\begin{proof} Fix $b_i \in M_+$ and $b \in M_+$ such that, for all $c \in M_+$, we have
\[
\langle c\xi_i, \xi_i \rangle = \tau_M(cb_i) \text{ and } \langle c\xi, \xi \rangle = \tau_M(cb).
\]
Fix $R>K$.  As in Lemma \ref{cutoff}, let $\pi_i^{(R)}$ denote the support projection of $b_i-b_i^{(R)}$.  Fix a non-principal ultrafilter $\cU$ on $\mathbf{N}$ and set $\pi^{(R)} := \lim_\cU \pi^{(R)}_i$.  Now observe that, for any $x \in M_+$, we have
\[
\begin{split}
R\tau_M(x\pi^{(R)}) = \tau(xR\pi^{(R)}) = \lim_\cU\tau(xR\pi_i^{(R)})& \leq \lim_\cU\tau(x\pi^{(R)}_ib_i) \leq \lim_\cU\tau(xb_i)\\
& = \tau(xb) \leq K\tau(x),
\end{split}
\]
whence it follows that $\tau(x\pi^{(R)}) \leq \displaystyle{\frac{K}{R}}\tau(x)$.  Setting $R := 2K$, we then have: for all $x \in M_+$, 
\[
\tau(x\pi^{(2K)}) \leq \displaystyle{\frac{1}{2}}\tau(x).
\]
Since a subsequence of $1 - \pi^{(2K)}_i$ converges weakly to $1 - \pi^{(2K)}$, we have that a subsequence of $(1-\pi^{(2K)}_i)\xi_i$ converges weakly to $(1-\pi^{(2K)})\xi$.  By Mazur's lemma\footnote{Mazur's lemma says that for normed linear spaces, for convex sets, being weakly closed and norm closed are the same.  See V.1.5 of \cite{Conway}}, we can form convex combinations $\mu_i$ from the sequence $(1-\pi^{(2K)}_i)\xi_i$ which converge in norm to $(1-\pi^{(2K)})\xi$.  By Lemma \ref{cutoff}(b), each $(1-\pi^{(2K)}_i)\xi_i$ is left $2K$-bounded.  By Lemma \ref{basic}, each $\mu_i$ is left $2K$-bounded and right bounded with right bound no larger than the maximum of the right bounds included in the convex combination.

\noindent{\bf Claim:} Suppose $\mu$ is $(L_1,L_2)$-bounded.  Then $\pi^{(2K)}\mu$ is $(L_1/2,L_2/2)$-bounded.

\noindent{\bf Proof of Claim:}  To see the right bound, suppose that $\psi\colon N \rightarrow M$ is a c.p.\ map such that $\overline{M\mu N} \cong  {_\psi H}$ via the map which sends $\mu$ to $1 \otimes 1$.  We compute the right bound of $\pi^{(2K)} \otimes 1$ in $_\psi H$: for $c \in N_+$,
\[
\begin{split}
\ip{\pi^{(2K)} \otimes c}{\pi^{(2K)} \otimes 1} &= \tau_M(\psi(c)\pi^{(2K)}) \leq \frac{1}{2}\tau_M(\psi(c))\\
& = \frac{1}{2}\ip{\mu c}{\mu} \leq \frac{L_2}{2}\tau_N(c).
\end{split}
\]

For the left bound, choose $d \in M_+$ such that $\ip{c\mu}{\mu} = \tau_M(cd)$ for all $c \in M_+$.  By Lemma \ref{basic}, $Md^{1/2}$ is isometric to $M\mu$ via the map sending $d^{1/2}$ to $\mu$, whence it suffices to check the left bound of $\pi^{2K}d^{1/2}$ in $Md^{1/2}$.  To this end, fix $c \in M_+$ and observe that
\[
\begin{split}
\ip{c\pi^{(2K)}d^{1/2}}{\pi^{(2K)}d^{1/2}} & = \tau_M(d^{1/2}\pi^{(2K)}c\pi^{(2K)}d^{1/2}) = \tau_M(\pi^{(2K)}c\pi^{(2K)}d)\\
& \leq L_1\tau_M(c\pi^{(2K)}) \leq \frac{L_1}{2}\tau_M(c).
\end{split}
\]

This concludes the proof of the claim.

We now apply the above procedure to the $\frac{K}{2}$-bounded vector $\pi^{(2K)}\xi$ and the sequence $\pi^{(2K)}\xi_i$, which is still a sequence of $K_i$-bounded vectors.  For convenience, we relabel items.  We have a projection $\pi^0$ and vectors $\eta_{00}, \eta_{10}, \eta_{20}, \ldots \in H_0$ which are $(2K,K_i)$-bounded and such that:
\begin{enumerate}
\item $\eta_{i0} \rightarrow (1 - \pi^0)\xi$ and
\item $\pi^0\xi$ is $\displaystyle\frac{K}{2}$-bounded.
\end{enumerate}

We then inductively create projections $\pi^i$ and sequences $\eta_{0i}, \eta_{1i}, \eta_{2i}, \ldots \in H_0$ such that
\begin{enumerate}
\item $\pi^i \leq \pi^{i-1}$, 
\item $\eta_{ji}$ is left $2^{1-i}\cdot K$-bounded and right-bounded,
\item $\eta_{ji} \rightarrow (\pi^{i-1} - \pi^i)\xi$, and 
\item $\pi^i\xi$ is $\displaystyle\frac{K}{2^{i+1}}$-bounded.
\end{enumerate}

Here, $\pi^{-1}:=1$.  By (2) and (3), $\|(\pi^{i-1} - \pi^i)\xi\|\leq 2^{1-i}\cdot K$, whence the series $\sum_{i=0}^{\infty}(\pi^{i-1} - \pi^i)\xi$ converges to $\xi$.  By choosing appropriate subsequences, we can create, by Lemma \ref{basic}, bounded vectors $\mu_i$ which tend to $\xi$ which are moreover left $B$-bounded, where
\[
B:=\left ( \sum_{i=0}^\infty \sqrt{\frac{2K}{2^i}} \right )^2 = \frac{2K}{(\sqrt{2}-1)^2}.
\]

We wish to repeat the argument above on the right.  There are only two places in the above argument where the bound on the right was effected.  The first is when we choose convex combinations of bounded vectors.  As noted above, if at the j$^{th}$ stage, all the vectors are uniformly right-bounded, then this will be true of the sequence $\eta_{0i},\eta_{1i},\eta_{2i},\ldots$.  Thus, when repeating the above argument on the right, at the i$^{th}$ stage, we can assume that the vectors are $\displaystyle\frac{B}{2^i}$-bounded on the left.

The other place where the right bound could be effected is in the creation of the final sequence.  If we knew that the right bound for each $\eta_{ji}$ was $\displaystyle\frac{B}{2^i}$, then the resulting sequence will have uniform upper bound on the right of
\[
\left ( \sum_{i=0}^\infty \left (\sqrt{\frac{B}{2^i}}\right ) \right )^2 = \frac{B}{(\sqrt{2}-1)^2}.
\]
Hence, after repeating the above process on the right, we find a sequence $\eta_i \rightarrow \xi$ such that each $\eta_i$ is $\left (\displaystyle\frac{B}{(\sqrt{2}-1)^2},B \right )$-bounded.  

The conclusion of the proposition now follows from Lemma \ref{KDT}.
\end{proof}

We can now state an ultraproduct version of the previous proposition.

\begin{prop}\label{mainpropultra}
Suppose that $\xi\in \prod_\cU H_i$ is $K$-bounded.  Then $\xi$ is uniformly $K$-bounded.
\end{prop}

\begin{proof} Without loss of generality, we may assume that $\xi=(\xi_i)^\bullet$ is  $1$-bounded and that each $\xi_i$ is bounded. Fix $R>1$ and set $\eta_i = (1-\pi_i^{(R)})\xi_i$, using all of the notations from the previous proposition.  Note that $\eta:=(\eta_i)^\bullet$ is bounded, uniformly left $R$-bounded, and, by the proof of the Claim in the previous proposition, we have $\|\eta-\xi\|\leq \sqrt{1/R}\|\xi\|$. Now perform the same procedure on the right to $\eta$, obtaining $\zeta$, which is now uniformly $R$-bounded and with $\|\zeta-\eta\|\leq \sqrt{1/R} \|\eta\|$, whence $\|\zeta-\xi\|\leq \sqrt{1/R}(1+\sqrt{1/R})\|\xi\|$.  
Letting $R\to \infty$, we see that $\xi$ is a limit of uniformly bounded vectors.  By Proposition \ref{mainprop}, we get that $\xi$ is a limit of uniformly bounded $K$-bounded vectors.  However, by Lemma \ref{KDTultra}, a uniformly bounded $K$-bounded vector is uniformly $K$-bounded; since the set of uniformly $K$-bounded vectors is closed, we get that $\xi$ is uniformly $K$-bounded, as desired.
\end{proof}

\begin{cor}\label{uniformdense}
The correspondence ultraproduct $\prod_{c\cU} H_i$ is the closure of the subspace consisting of the uniformly bounded vectors.
\end{cor}

\subsection{Freeing the von Neumann algebras}
We now remove the assumption that we consider correspondences over a fixed pair of algebras.  Instead, consider a family $(M_i,H_i,N_i)$, where $M_i$ and $N_i$ are tracial von Neumann algebras and $H_i$ is an $M_i$-$N_i$ correspondence, and an ultrafilter $\cU$ on $I$.  Set $M:=\prod_\cU M_i$, $N:=\prod_\cU N_i$, and $H:=\prod_\cU H_i$.  Unlike the case when the von Neumann algebras are fixed, it is not even clear that there are well-defined maps $(a,\xi)\mapsto a\xi$ and $(b,\xi)\mapsto \xi b$ from $M\times H$ and $N\times H$ to $H$.  More precisely, if $(a_i)^\bullet=(a_i')^\bullet$ and $(\xi_i)^\bullet=(\xi_i')^\bullet$, it is not necessarily true that $(a_i\xi)^\bullet=(a_i'\xi_i')^\bullet$.  Thus, in this case, we call $\xi\in H$ \textbf{a point of continuity} if:
\begin{enumerate}
\item the assignments $a\mapsto a\xi$ and $b\mapsto \xi b$ are well-defined, that is, independent of the choice of representatives for $a$, $b$, and $\xi$;
\item the maps $a\mapsto a\xi:M\to H$ and $b\mapsto \xi b:N\to H$ extend to bounded linear maps on $L^2(M,\tau_M)$ and $L^2(N,\tau_N)$.
\end{enumerate}

\begin{lem}
The closure of the set of points of continuity in $H$ form an $M$-$N$ correspondence, called the \textbf{correspondence ultraproduct} of the family $(M_i,H_i,N_i)$, denoted once again by $\prod_{c\cU}H_i$.
\end{lem} 

As before, we call $\xi\in H$ \textbf{uniformly $K$-bounded} if we can write $\xi=(\xi_i)^\bullet$ with each $\xi_i\in H_i$ $K$-bounded.

\begin{lem}

\

\begin{enumerate}
\item If $\xi\in H$ is uniformly $K$-bounded, then $\xi$ is a point of continuity and is a $K$-bounded vector of $\prod_{c\cU}H_i$.
\item The set of uniformly $K$-bounded vectors forms a closed set.
\item The set of uniformly bounded vectors is a subspace of $H$ invariant under the actions of $M$ and $N$.
\end{enumerate}
\end{lem}

\begin{remark}
In the case that $M_i=M$ and $N_i=N$ for all $i$, it seems that we might have conflicting notation for the correspondence ultraproduct currently being discussed and the construction in the previous subsection.  We explain this abuse of notation now.  For the sake of discussion, let $\prod_{c\cU}^1H_i$ denote the correspondence ultraproduct from the previous subsection  and $\prod_{c\cU}^2H_i$ denote the correspondence ultraproduct from this subsection.  By the previous lemma, every uniformly $K$-bounded element of $\prod_{c\cU}^1H_i$ belongs to $\prod_{c\cU}^2H_i$; by Corollary \ref{uniformdense}, it follows that $\prod_{c\cU}^1H_i\subseteq \prod_{c\cU}^2H_i$.  However, it follows immediately from the definitions that $\prod_{c\cU}^2H_i\subseteq \prod_{c\cU}^1H_i$.  As a result, $\prod_{c\cU}^1H_i=\prod_{c\cU}^2H_i$ as Hilbert spaces.  In other words, when dealing with $M$-$N$ correspondences, the correspondence ultraproduct from Section \ref{corrultra} is simply the result of viewing the correspondence ultraproduct of this subsection, which is an $M^\cU$-$N^\cU$-correspondence, as merely an $M$-$N$ correspondence.
\end{remark}

As before, we have:

\begin{thm}\label{freeultra}
If $\xi\in H$ is $K$-bounded, then $\xi$ is uniformly $K$-bounded.  Consequently, $\prod_{c\cU} H_i$ is the closure of the subspace consisting of the uniformly bounded vectors.
\end{thm}

\begin{proof}
One needs merely to observe that the proofs in the previous subsection go through in this context as well.
\end{proof}

\subsection{Connection to Connes' ultraproduct}

In Section I.3 of his influential paper \cite{connes76}, Connes pointed out that the tracial ultraproduct of a family of factors standardly represented does not act on the ultraproduct of the Hilbert spaces and instead proposed that one consider a subspace of the Hilbert space ultraproduct on which the tracial ultraproduct of the factors does act.  It is the purpose of this subsection to show that the ultraproduct consider there which we will call the Connes' ultraproduct is a special case of our ultraproduct in the case that one is considering the trivial correspondences.  (Technically, \cite{connes76} only considers left modules rather than bimodules, but this does not affect anything that we will say below.)  We outline the construction from \cite{connes76}.

Let $(M_i  : i\in I)$ be a family of finite factors and let $\cU$ be an ultrafilter on $I$ with ultraproduct $M:=\prod_\cU M_i$.  (\cite{connes76} only considers countable families, but what is done there works in general.)  Set $H_i:=L^2(M_i)$ and $H:=\prod_\cU H_i$.  One now considers the subspace ${}_{\cU}\prod H_i$ of $H$ consisting of those vectors $\xi=(\xi_i)^\bullet$ for which:  given any $\epsilon>0$, there is $K>0$ such that 
\[
\lim_\cU \|\chi_K(|\xi_i|) |\xi_i|\|_2<\epsilon.
\]
  Here, $\chi_K$ denotes the characteristic function of the open interval $(K,\infty)$ in $\mathbb{R}$.  Since we can assume that $\xi_i \in M_i$, the meaning of $\chi_K(|\xi_i|)$ is as in the Borel functional calculus on $M_i$. One shows that ${}_{\cU}\prod H_i$ is a closed subspace of $H$ and $M$ acts on ${}_{\cU}\prod H_i$ in the standard way.

\begin{prop}
Connes' ultraproduct ${}_{\cU}\prod H_i$ is precisely equal to the correspondence ultraproduct $\prod_{c\cU}H_i$.
\end{prop}

\begin{proof}
By the maximality of the correspondence ultraproduct, we have that ${}_{\cU}\prod H_i$ is contained in $\prod_{c\cU}H_i$.  For the reverse inclusion, it suffices to note that the uniformly bounded elements of $\prod_{c\cU}H_i$ belong to ${}_{\cU}\prod H_i$.  However, as noted in Example \ref{boundedstandard}, if $\xi$ is uniformly $K$-bounded, then $\xi=(\xi_i)^\bullet$, with  each $\xi_i\in M_i$ satisfying $\xi_i^*\xi_i\leq K\cdot 1$.  It is clear that for any $L>K^2$, we have that $\chi_L(|\xi_i|)=0$, whence $\xi$ belongs to ${}_{\cU}\prod H_i$.
\end{proof}

%

\section{The elementary class of correspondences}

\subsection{Elementarity} In all but the final subsection of this section, we fix tracial von Neumann algebras $(M,\tau_M)$ and $(N,\tau_N)$.  The language $\cL=\cL_{M,N}$ of $M$-$N$ correspondences will consist of:
\begin{enumerate}
\item For each $K\in \bbN$, there is a sort $S_K$.  The metric $d_K$ on $S_K$ is assumed to have bound $K$.
\item For each $K < L$, there is a function symbol $i_{KL}\colon S_K \rightarrow S_L$ intended to be an isometric embedding.
\item For each $K$ and $L$, there is a binary function symbol $+$ (we suppress the dependence on the pair $(K,L)$) which is 1-Lipschitz in each coordinate and maps $S_K \times S_L$ into $S_M$, where $M$ is the ceiling of $(\sqrt{K} + \sqrt{L})^2$.
\item For each $K$, there is a binary relation $\ip{\cdot}{\cdot}$ on each $S_K$ (again we will suppress the dependence on $K$).  It will be 1-Lipschitz in each coordinate and bounded by $K$.  (Technically, there should really be two such function symbols, one for the real part and one for the imaginary part of the inner product.)
\item For each $c \in M$, there is a unary function symbol on $S_K$ (again avoiding the dependence on $K$).  We will write $cx$ for the application to an $S_K$-sorted variable.  This function maps into $S_M$, where $M$ is the ceiling of $\|c\|^2K$, and is $\|c\|$-Lipschitz.
\item Likewise, for each $c \in N$, there is a unary function symbol on $S_K$.  We will write $xc$ for the application to an $S_K$-sorted variable.  This function maps into $S_M$, where $M$ is the ceiling of $\|c\|^2K$, and is $\|c\|$-Lipschitz.
\end{enumerate}

Next, for each $M$-$N$ correspondence $H$, we describe an $\cL$-structure $\M(H)$ as follows:

\begin{enumerate}
\item For each $K\in \bbN$, set $S_K(\M(H))$ to be the set of $K$-bounded vectors in $H$.  Let $d_K$ be the metric on $S_K(\M(H))$ induced by the inner product from $H$.  (The Cauchy-Schwartz inequality guarantees the required bound on the metric.)
\item For each $K < L$, $i_{KL}$ is the inclusion map from $S_K(\M(H))$ to $S_L(\M(H))$.
\item For each pair $(K,L)$, addition on $S_K(\M(H)) \times S_L(\M(H))$ is the restriction of addition on $H$.  (The range is justified by Lemma \ref{basic}.)
\item For each $K$, $\ip{\cdot}{\cdot}$ on $S_K(\M(H))$ is the restriction of the inner product on $H$.
\item For $c \in M$, the interpretation of $c$ on $S_K$ is the restriction of the left action on $H$.  Likewise, for $c\in N$, it is the restriction of the right action.  (Once again, the range is justified by Lemma \ref{basic}.)
\end{enumerate}

We refer to $\M(H)$ as the \textbf{dissection} of $H$.  Here is the model-theoretic reformulation of Proposition \ref{mainpropultra}:

\begin{thm}\label{coincides}
Suppose that $(H_i \ : \ i\in I)$ is a family of $M$-$N$ correspondences and that $\cU$ is an ultrafilter on $I$.  Then $\M(\prod_{c\cU}H_i)=\prod_\cU \M(H_i)$ (where the latter ultraproduct is the usual model-theoretic ultraproduct.)
\end{thm}

Let $\cal C:=\cal C_{M,N}$ be the category of $\cL$-structures consisting of dissections of $M$-$N$ correspondences with $\cL$-homomorphic embeddings as morphisms.  Let $\operatorname{Corr}(M,N)$ be the category of $M$-$N$ correspondences with isometric correspondence maps as morphisms.
\begin{thm}\label{mainthm}
$\cal C$ is an elementary class.  Moreover, $\cal C$ is equivalent as a category to $\operatorname{Corr}(M,N)$.
\end{thm}
\proof We first show that $\cal C$ is closed under ultraproducts and ultraroots.  Closure under ultraproducts follows immediately from Theorem \ref{coincides}.
Now suppose that $\M$ is an $\cL$-structure such that $\M^\cU = \M(H)$ for some correspondence $H$.  It is clear that the closure of the union of the sorts $S_K(\M)$ is a sub-correspondence $H_0$ of $H$.  It remains to see that $\M=\M(H_0)$.  To see this, suppose that $\xi$ is a $K$-bounded vector in $H_0$.  Then $\xi$ is a $K$-bounded vector in $H$, whence $\xi \in S_K(M^\cU)$.  Fix $\epsilon>0$ and choose $\xi'\in S_L(M)$ such that $\|\xi-\xi'\|\leq \epsilon$.  Since $\M$ is an elementary substructure of $\M^\cU$, we have that
$$\inf_{\eta\in S_K(\M)}d(\xi',\eta)=\inf_{\eta\in S_K(\M^\cU)}d(\xi',\eta)\leq \epsilon.$$
Thus, there is $\eta\in S_K(\M)$ such that $d(\xi,\eta)\leq 3\epsilon$.  Since $\epsilon$ was arbitrary, we conclude that $\xi\in S_K(\M)$.  

We now prove the categorical equivalence. Given $\M\in \cal C$, we let $H(\M)$ denote the corresponding correspondence, which is simply the completion of the union of the $S_K(\M)$'s.  It is clear that $H(\M(H)) = H$ for any correspondence $H$.
We now show that $\M(H(\M))=\M$ for every $\M\in \cal C$.  To see this, suppose that $\xi$ is a $K$-bounded vector in $H(\M)$.  Let $H_0$ denote the union of the sorts of $H(\M)$, which is a dense subspace of $H(\M)$ closed under the actions of $M$ and $N$.  It follows that $\xi=\lim_n \xi_n$, each $\xi_n\in H_0$.  Since $H_0$ consists of bounded vectors, by Proposition \ref{mainprop}, we may write $\xi=\lim_n \xi_n'$, each $\xi_n'\in H_0$ $K$-bounded.  Since $\M$ is isomorphic to the dissection of a correspondence, we have that an element of $S_L(\M)$ that is $K$-bounded is in the image of $i_{KL}$.  It follows that $\xi\in S_K(\M)$, as desired.  

The fact that the functors $\M$ and $H$ provide a bijection between the Hom-sets is straightforward and so we conclude that we have an equivalence of categories.
\qed

\begin{remark}
Recall that, given structures $\M$ and $\cal N$ with $\cal N\subseteq \M$, we say that $\cal N$ is \emph{existentially closed (e.c.) in $\M$} if any existential statement with parameters from $\cal N$ has the same value in both $\cal N$ and $\M$.  The proof of the previous result shows that an e.c.\ substructure of a model of $T_{\corr}$ is once again a model of $T_{\corr}$.
\end{remark}

Let $T_{\corr}:=T_{M\text{-}N \corr}$ denote the theory axiomatizing the class of dissections of $M$-$N$ correspondences. 

\begin{cor}
$T_{\corr}$ is $\forall\exists$-axiomatizable.
\end{cor}

\begin{proof}
By a well-known test, it suffices to check that an e.c.\  substructure of a model of $T_{\corr}$ is a model of $T_{\corr}$.  This fact was already pointed out in the previous remark.
\end{proof}

We can in fact give a concrete axiomatization of $T_{\corr}$.  We first list the more straightforward axioms:
\begin{itemize}
\item Axioms for (dissections of) Hilbert spaces.
\item Axioms saying that the elements of $M$ and $N$ act as bounded operators.  To express the boundedness, we have, for each $a\in M$ and each $K$, the axiom
$$\sup_{x\in S_K} (\|ax\|\dminus \|a\|\cdot \|x\|)=0$$ and similar axioms for elements of $N$.
\item Axioms saying that $\pi_l$ and $\pi_r$ are $*$-homomorphisms. 
\item Axioms saying that $\pi_l$ and $\pi_r$ commute:  for each $a\in M$, $b\in N$ and each $K$, we have the axiom
$$\sup_{x\in S_K}d(a(xb),(ax)b)=0.$$
\item Axioms saying that elements of $S_K$ are $K$-bounded:  for each $a\in M$, we have the axiom
$$\sup_{x\in S_K}\left(\langle ax,x\rangle\dminus K\tau_M(a)\right)=0$$ and similar axioms for elements of $N$.
\end{itemize}

There is one last axiom scheme which expresses that the elements of $S_K$ are precisely the $K$-bounded elements.  Informally, this amounts to saying that if $K<L$ and $y\in S_L$ is $K$-bounded, then there is $x\in S_K$ such that $i_{KL}(x)=y$.  In order to write this precisely, we first let $X_{KL}$ denote the $T_{\corr}$-functor that collects the set of $K$-bounded vectors in $S_L$.  By Lemma \ref{KDTultra}, $X_{KL}$ is a $T_{\corr}$-definable set (being preserved under ultraproducts may be taken as a definition of definable set), whence we may thus express our last axiom scheme as follows:

\begin{itemize}
\item For each $K<L$, we have the axiom
$$\sup_{y\in X_{KL}}\inf_{x\in S_K}i_{KL}(x)=y.$$
\end{itemize}

In order to write the above axioms in their unabbreviated form, one would need to use some $T_{\corr}$-definable predicate $\varphi_{KL}$ for which $X_{KL}$ equals the zeroset $Z(\varphi_{KL})$ of $\varphi_{KL}$.  Since we proved that $X_{KL}$ is definable by showing that it was preserved by ultraproducts, the definable predicate $\varphi_{KL}$ comes to us via theoretical means, seemingly making the above axiomatization less concrete.  However, this is readily remedied as follows.  Given $\epsilon>0$, the union of the following collection of conditions is unsatisfiable:
\begin{itemize}
\item $\{\ip{ax}{x}\dminus K\tau_M(a)=0 \ : \ a\in M\}$;
\item $\{\ip{xb}{x}\dminus K\tau_N(b)=0 \ : \ b\in N\}$;
\item $\{\epsilon\dminus \varphi_{KL}(x)\}.$
\end{itemize}

By compactness and letting $\epsilon$ range over all positive rational numbers, we find countable sequences $(a_n)$ and $(b_n)$ from $M$ and $N$ respectively such that, setting $\psi_{KL}$ to be the $T_{\corr}$-definable predicate $$\psi_{KL}(x):=\sum_n 2^{-n}\max(\ip{a_nx}{x}\dminus K\tau_M(a_n),\ip{xb_n}{x}\dminus K\tau_N(b_n)),$$ we have that $Z(\varphi_{KL})=Z(\psi_{KL})$.  Thus, in the unabbreviated form of the last axiom scheme, one can work with the arguably more concrete representation of $X_{KL}$ as $Z(\psi_{KL})$.

Note that only the last axiom scheme was $\forall\exists$, the prior axioms all being universal.  There is in fact a language for which the class of (definitional expansions of) dissections of correspondences becomes a universally axiomatizable class.  Since $S_1$ is convex, we can consider the functions $f_K\colon S_K\to S_1$ which maps $a \in S_K$ to the closest element in $S_1$.  As the functions $f_K$ are preserved by ultraproducts, we have that these functions are $T_{\corr}$-definable.  
We add these symbols to the above language and let $\cal C^*$ denote the corresponding class of expansions of elements of $\cal C$.  Since these are definitional expansions, $\cal C^*$ is an axiomatizable class, say the models of $T_{\corr}^*$.  

\begin{prop}
$T_{\corr}^*$ is universally axiomatizable.
\end{prop}

\begin{proof}
In order to check that $T_{\corr}^*$ is universally axiomatizable, it suffices to show that a substructure of a model of $T_{\corr}^*$ is once again a model of $T_{\corr}^*$.  (See, for example, \cite[Proposition 2.4.4]{munster}.)  Suppose that $\mathcal{N}\subseteq \mathcal{M}\models T_{\corr}^*$.  The main point is to check that if $a\in H(\mathcal{N})$ is 1-bounded, then $a\in S_1(\mathcal{N})$.  Since $\mathcal{M}\models T$, we have that $a\in S_1(\mathcal{M})$.  Let $a_n\in S_{K_n}(\mathcal N)$ for $n \in\bbN$ be such that $a_n\to a$.  Then, in $\mathcal{M}$, we have that $\|a-f_{K_n}(a_n)\|\leq \|a-a_n\|$, whence $f_{K_n}(a_n)\to a$.  But $f_{K_n}(a_n)\in S_1(\mathcal{N})$ for all $n$, whence $a\in S_1(\mathcal{N})$.
\end{proof}

\subsection{The Fell topology and weak containment}

Recall that the \textbf{Fell topology} is the topology on the space of (isomorphism classes of) correspondences whose basic open neighbourhoods of a correspondence $H$ are given by $V(H;\epsilon,E,F,S)$, where $E$ and $F$ are finite subsets of $M$ and $N$ respectively, $S=\{\xi_1,\ldots,\xi_n\}$ is a finite subset of $H$, and $V(H;\epsilon,E,F,S)$ consists of those correspondences $K$ for which there are $\eta_1,\ldots,\eta_n\in K$ satisfying
$$\left| \langle \xi_i,x\xi_jy\rangle - \langle \eta_i,x\eta_jy\rangle\right|<\epsilon$$ for all $1\leq i,j\leq n$ and all $x\in E$ and $y\in F$.

If $H$ and $K$ are correspondences, we say that $H$ is \textbf{weakly contained} in $K$ if $H$ is in the closure of $\{K^{\oplus\infty}\}$ (in the Fell topology).  We use our above analysis to give an ultrapower reformulation of weak containment analogous to that appearing in the theory of unitary group representations.

\begin{prop}
Given correspondences $H$ and $K$, the following are equivalent:
\begin{enumerate}
\item $H$ is in the closure of $\{K\}$.
\item $\M(H)\models \Th_\forall(\M(K))$.
\item $\M(H)$ embeds into $\M(K)^{\cU}$.
\item $H$ embeds into $K^{c\cU}$.
\end{enumerate}
\end{prop}

\begin{proof}
The equivalence of (2) and (3) is standard and the implications (3) implies (4) implies (1) are straightforward.  We now prove that (1) implies (3).  We may assume that $H$ is separable.  For simplicity, we also assume that $M$ and $N$ are separable.  It suffices to show that $S_1(H)$ embeds into $S_1(K)^\cU$.  Let $(\xi_i)$ enumerate a countable dense subset of $S_1(H)$.  Let $(E_n)$ and $(F_n)$ denote increasing sequences of finite subsets of $M_1$ and $N_1$ respectively with dense unions. Since $H$ is in the closure of $K$, we may find, for $i\leq n$, vectors $\eta_i^n\in \bigcup_m S_m(K)$ such that

$$\left| \langle \xi_i,x\xi_jy\rangle - \langle \eta_i^n,x\eta_j^ny\rangle\right|<\epsilon$$ for all $1\leq i,j\leq n$ and all $x\in E_n$ and $y\in F_n$.  Set $\eta_i:=(\eta_i^n)^\bullet\in K^{\cU}$.  Note that $\eta_i$ is $1$-bounded, whence, by Proposition \ref{mainpropultra}, we have that $\eta_i\in S_1(K)^{\cU}$ for each $i$.  The  map $\xi_i\mapsto \eta_i$ extends to the desired embedding of $S_1(H)$ into $S_1(K)^{\cU}$.
\end{proof}


\subsection{Classification, stability, and model companions}
Roughly speaking, we say that a theory is \textbf{classifiable} if there is a generalized set of dimension functions which characterize the isomorphism type of its models; see \cite[Chapter 13]{Shelah} for more details than are needed here. In the case of $M$-$N$ correspondences, the dimensions are quite easy in light of Corollary \ref{dirsum}.  Indeed, since every correspondence is a direct sum of cyclic correspondences associated to subtracial c.p. maps, and the number of those, up to isomorphism, is bounded, 
 a model of $T_{\corr}$ is determined by specifying the number of each direct summand present in a direct sum.  We have just shown:

\begin{thm}
$T_{\corr}$ is classifiable.
\end{thm}

One consequence of classifiability is that the theory in question is stable.  In this case, one expects all completions of $T_{\corr}$ to be  superstable, which is indeed the case:

\begin{thm}
Suppose that $T$ is a completion of $T_{\corr}$.  Then $T$ is superstable.
\end{thm}  

\begin{proof}
To see this, suppose that $T$ is a completion of $T_{\corr}$ and $\M\models T$.  Every type over $M$ in some fixed finitely many variables generates a correspondence of the form $H(\M) \oplus K$ for some correspondence $K$ which has density character $\mu$, the maximum of that of $M$, $N$ and $\aleph_0$.  Using Corollary \ref{dirsum} and some very crude counting, if the density character of $\M$ is $\lambda$ with $\lambda \geq 2^\mu$, then $T$ is $\lambda$-stable.  
\end{proof}

By virtue of being stable, there exists a well-behaved independence relation, which we now concretely describe for completions of the theory of correspondences.  We work in the language of $T_{\corr}^*$ and fix a completion $T$.  We work in a large, saturated model $H_T$ of $T$ and set $A_0$ to be the algebraic closure of the empty set.  

Before describing the independence relation, we note the following:

\begin{prop}\label{QE}
Let $T'$ be the theory $T$ with constants added to name elements of $A_0$.  Then:
\begin{enumerate}
\item $T'$ has quantifier elimination.
\item The algebraic closure of any $X \subseteq H_T$ is the correspondence generated by $X$ and $A_0$.
\end{enumerate}
\end{prop}

\begin{proof} For (1), suppose that $A$ is a separable correspondence containing $A_0$ and contained in $H_0$ and $H_1$, both models of $T$.  Moreover, assume that $H_0$ is separable and $H_1$ is $\aleph_1$-saturated.  Any direct sum decomposition of $A$ can be extended to a direct sum decomposition of $H_0$.  Suppose that $P$ is one of the subtracial cyclic direct summands of $H_0$ orthogonal to $A$.  Since $H_1$ is $\aleph_1$-saturated, it contains a copy of $P^{\oplus \aleph_1}$.  This family of direct summands can be extended to a direct sum decomposition of $H_1$. Since $A$ is separable, it is contained in the direct sum of only countably many summands that make up this decomposition of $H_1$ and so at least one copy of $P$ must be orthogonal to $A$.  This means that one can fix $A$ and map $P$ into $H_1$. Proceeding like this summand by summand, we get an embedding of $H_0$ into $H_1$ fixing $A$. Quantifier elimination follows. 

For (2), let $X'$ be the correspondence generated by $X$ and $A_0$ and consider $\xi \not\in X'$.  Without loss of generality, we can assume that $\xi$ is orthogonal to $X'$ and that $\xi$ is subtracial.  If $P$ is the cyclic correspondence that $\xi$ generates, we know as above that $P^{\oplus \lambda}$ is also present in $H_T$ for any small $\lambda$.  By quantifier elimination, the type of $\xi$ has unboundedly many realizations over $X'$, whence $\xi$ is not algebraic over $X'$ (and thus not algebraic over $X$).
\end{proof}

We can now describe the independence relation.  For subcorrespondences $A, B$ and $C$ of $H$, where $A_0 \subseteq A \subseteq B\cap C$, we say that \textbf{$B$ is independent from $C$ over $A$} if the correspondence generated by $B$ and $C$ has the form $A \oplus B_0 \oplus C_0$, where $B = A \oplus B_0$ and $C = A \oplus C_0$.

\begin{prop} The independence relation described above is the non-forking relation for $T$.
\end{prop}

\begin{proof}
This relation is easily seen to be invariant, transitive and symmetric.  If $B$ and $C$ are as above and are dependent over $A$, then this is explained by some $a \in B \cap C$, whence this relation has finite character.  If $B$ is separable, then $B$ is independent from $C$ over $B \cap C$, whence this relation also has local character.

The only remaining properties of non-forking we need to check are extension and stationarity.  They are similar and so we just check extension.  Suppose that $A_0 \subseteq A \subseteq B, C$ are subcorrespondences in $H$.  We may assume that $B = B_0 \oplus A$.  As in the proof of quantifier elimination, we can find some isomorphic copy $B'_0$ of $B_0$ which is disjoint from $C$.  Then $B' = B'_0 \oplus A$ is independent from $C$ over $A$, which proves extension.
\end{proof}

In certain cases, we can improve the previous result.  We call a complete theory $T$ of correspondences \textbf{ample} if whenever $P$ is a cyclic, subtracial correspondence contained in $H_T$, then $P^{\oplus\aleph_0} \subseteq H_T$.  Observe that the proof of Proposition \ref{QE} shows that ample theories have quantifier elimination in the language of $T_{\corr}^*$.

\begin{question}
Is every complete theory of correspondences ample?  In particular, when the is the theory of the trivial $M$-$M$ correspondence ample?
\end{question}

Let $\cal S$ denote the set of subtracial c.p. maps $M\to N$.  Given a closed (in the weak topology) subset ${\cal A}$ of $\cal S$, we say that ${\cal A}$ is \textbf{ample} if whenever $\varphi_i \in {\cal A}$ for $i \in \bbN$ (possibly with repetitions) and $H_\psi \subseteq \oplus_i H_{\varphi_i}$, then $\psi \in {\cal A}$.

To an ample set of subtracial c.p.\  maps ${\cal A}$, we associate the theory $T_{\cal A}$ of the correspondence $\bigoplus_{\varphi \in \cal A} H_\varphi^{\oplus \aleph_0}$.


\begin{prop} For a complete theory $T$ of correspondences, the following are equivalent:
\begin{enumerate}
\item $T$ is ample.
\item $T = T_{\cal A}$ for some ample $\cal A$.
\item The algebraic closure of the empty set is $\{0\}$.
\end{enumerate}
\end{prop}

\begin{proof}
The only direction that needs proof is (1) implies (2).  Let $\cal A$ be the set of all subtracial c.p. maps $\varphi$ such that $H_\varphi \subseteq H_T$ and let ${\cal A}_0$ be a countable dense subset of $\cal A$.  Since $T$ is ample, we can build a copy of $\bigoplus_{\varphi \in {\cal A}_0} H_\varphi$ in $H_T$.  By quantifier elimination, this is an elementary submodel of $H_T$ and a model of $T_{\cal A}$.
\end{proof}

 The largest such ample set is, of course, $\cal S$, the set of all subtracial c.p. maps and the theory $T_{\cal S}$ has the property that every correspondence embeds into a model of it.  Thus, we have:


\begin{thm}
The theory of $M$-$N$ correspondences has a model companion, namely $T_{\cal S}$, with quantifier elimination in the language of $T_{\corr}^*$.
\end{thm}


There is a slightly different approach to the model-companion of $T_{\corr}$ that follows the lines of Berenstein's proof that the theory of unitary representations of a fixed group $\Gamma$ has a model companion (see \cite{berenstein}).  We merely summarize the appropriate sequence of lemmas, leaving the reader to check that the proofs are identical to those in \cite{berenstein}.

Suppose that $H$ is any separable \textbf{locally universal} correspondence, that is, one whose ultrapower embeds any separable correspondence.  (Such correspondences exist since $T_{\corr}$ is $\forall\exists$-axiomatizable; for instance, let $H$ be an e.c.\ correspondence.)  

\begin{lem}\label{oplusec}
Suppose that $K$ is a correspondence.  Then $K\oplus H^{\oplus\infty}$ is e.c.
\end{lem}

\begin{lem}
Suppose that $K$ is a correspondence.  Then $K$ is e.c.\ if and only if $K\models \Th_\exists(H^{\oplus\infty})$.
\end{lem}

\begin{cor}
$T_{\corr}$ has a model companion, namely $T_{\corr}\cup \Th_\exists(H^{\oplus\infty})$.  Moreover, this model companion is a model completion in the language of $T^*_{\corr}$.
\end{cor}

The group version of the following corollary also appears in \cite{berenstein}.

\begin{cor}
$M$ is amenable if and only if the infinite direct sum of the coarse correspondence $L^2(M,\tau_M)\otimes L^2(M,\tau_M)$ is e.c.
\end{cor}

\begin{proof}
By Proposition \ref{effroslance} (and the fact that amenability coincides with semidiscreteness), $M$ is amenable if and only if the infinite direct sum of the coarse correspondence is locally universal; the latter condition is equivalent to the infinite direct sum of the correspondence being e.c.\ by Lemma \ref{oplusec}.
\end{proof}


\subsection{Freeing the von Neumann algebras (continued)}

In this subsection, we revisit the class of correspondences but this time do not insist that the tracial von Neumann algebras are fixed.  
The class of structures $\operatorname{Corr}$ we are trying to capture now are triples $(M,H,N)$, where $M$ and $N$ are tracial von Neumann algebras and $H$ is an $M$-$N$ correspondence.  We introduce a language $\cL_{\mathcal{F}}$ as follows:

\begin{enumerate}
\item For $M$ and $N$, we use the language of tracial von Neumann algebras.
\item For $H$, we have sorts $S_K$, which, as above, are meant to capture the $K$-bounded vectors.  We also have an inner product and $+$ just as in the case of fixed algebras.
\item Additionally we will have actions from $M \times H$ and $H \times N$ mapping to $H$ with the usual restrictions on domains, ranges and continuity moduli.
\end{enumerate}

For any $(M,H,N) \in \operatorname{Corr}$, we can once again consider its dissection, which is the $\cL_{\mathcal{F}}$-structure $\M(M,H,N)$, which assigns all the sorts, functions and relations their standard intended meaning.  Let ${\cal F}$ be the class of $\cL_{\mathcal{F}}$-structures $\M(M,H,N)$ for all $(M,H,N) \in \operatorname{Corr}$.  The model-theoretic version of Theorem \ref{freeultra} is the following:

\begin{thm}\label{modelultra}
Suppose that $(M_i,H_i,N_i)$ is a family of elements of $\operatorname{Corr}$ and $\cU$ is an ultrafilter on $I$.  Then $\M(\prod_\cU M_i,\prod_{c\cU}H_i, \prod_\cU N_i)=\prod_\cU \M(M_i,H_i,N_i)$.
\end{thm}

Given Theorem \ref{modelultra}, the proof of the next theorem is identical to its ``fixed'' analogue.
\begin{thm}
$\cal F$ is an elementary class which is categorically equivalent to $\operatorname{Corr}$.
\end{thm}


\section{Property (T)}
\setcounter{subsection}{1}
In \cite{Gold}, the first named author showed that a countable group $\Gamma$ has property (T) if and only if the set of $\Gamma$-invariant vectors is a definable set relative to the theory of unitary representations of $\Gamma$.  One of the main motivations for the current paper is to prove an analogous result for II$_1$ factors, which is the content of this section.

 Let us first recall the definition of property (T) for tracial von Neumann algebras.  While we follow \cite{Popa-book}, the definition there is in terms of c.p.\ maps.  Instead, we give an equivalent formulation, which is the content of \cite[Proposition 14.2.4]{Popa-book}.  Given an $M$-$M$ correspondence $H$, $a\in M$, and $\xi\in H$, we set $[a,\xi]:=a\xi-\xi a$.  We say that $\xi\in H$ is \textbf{central} if $[a,\xi]=0$ for all $a\in M$.

\begin{defn}\label{Tdef}
$(M,\tau_M)$ has \textbf{property (T)} if, for any $\epsilon>0$, there is a finite subset $F\subseteq M$ and $\delta>0$ such that, for any $M$-$M$ correspondence $H$ and any tracial vector $\xi\in H$ satisfying $\max_{x\in F}\|[x,\xi]\|\leq \delta$, there is a central vector $\eta\in H$ with $\|\xi-\eta\|\leq \epsilon$. 
\end{defn}

For us, the following result (\cite[Proposition 14.5.1]{Popa-book}) is crucial:

\begin{prop}
If $M$ is a II$_1$ factor, then in the above definition of property (T), one may replace the word ``tracial'' with ``unit.''
\end{prop}


For $\M\models T_{\corr}$, let $X(\M)$ denote the set of elements of $S_1(\M)$ that are central vectors of $H(\M)$.  The goal of this section is to prove, when $M$ is a II$_1$ factor, that $X$ is a $T_{\corr}$-definable set if and only if $M$ has property (T).  We first need one technical lemma:

\begin{lem}\label{Tlem} Let $H$ be an $M$-$M$ correspondence and suppose that $\xi\in H$ is a $K$-bounded unit vector with $\sup_{u\in\cal U(M)}\|[u,\xi]\|\leq \delta$. Then there is a $K/(1-\delta)^2$-bounded \emph{central} unit vector $\eta\in\cal H$ with $\|\xi - \eta\|\leq 2\delta$.
\end{lem}

\begin{proof} Since $\xi$ is a $K$-bounded vector, so is $u\xi u^*$ for any $u\in\cal U(M)$, whence so is any vector in $S :=\overline{\rm co}\{u\xi u^* : u\in \cal U(M)\}$. Since $S$ is closed and convex, there is a unique vector $\eta'\in \cal S$ which is closest to $0$. Since $uSu^* = S$ for all $u\in\cal U(M)$, this implies that $\eta'$ is central. We have that $\|\xi - \eta'\|\leq \delta$, so $\|\eta'\|\geq 1-\delta$ by the triangle inequality. Since $\eta'$ is $K$-bounded, we have that $\eta = \eta'/\|\eta'\|$ is $K/(1-\delta)^2$ bounded. Moreover, $\|\xi -\eta\|\leq 2\delta$. \qedhere

\end{proof}


\begin{thm}
Suppose that $M$ is a II$_1$ factor.  Then $M$ has property (T) if and only if the $T_{\corr}$-functor $X$ is a definable set.
\end{thm}

\begin{proof}  The ``if'' direction is immediate from the definitions and does not use the assumption that $M$ is a II$_1$ factor.

Now suppose that $M$ has property (T).  It suffices to show that:  for any family $(H_i \ : \ i\in I)$ of $M$-$M$ correspondences and any ultrafilter $\cU$ on $I$, that $\prod_\cU X(H_i) = X(\prod_{c\cU} H_i)$.  The inclusion from left to right is clear.  Now consider $\xi=(\xi_i)^\bullet \in X(\prod_{c\cU} H_i)$.  By scaling by its length, we can assume that $\xi$ is a unit vector and $K$-bounded.  For each $n \in \bbN$, choose finite sets $F_n \subseteq U(M)$ with dense union (recall that if $M$ has property (T) then $M$ is separable) and a decreasing sequence $\delta_n \leq 1/(n+1)$ witnessing that $M$ has property (T) for $\epsilon:=\frac{1}{n}$.

Choose sets $U_n := \{ i \in I : \max_{u \in F_n} \|[u,\xi_i]\| \leq \delta_n \}$ in $\cU$ and which are decreasing with $n$.  This is possible since $\xi$ is central.  Define $\xi'_i\in H_i$ as follows: for $i \in U_n \setminus U_{n+1}$, let $\xi'_i \in H_i$ be central such that $\|\xi'_i - \xi_i\| \leq 1/n$; by scaling a little using Lemma \ref{Tlem}, we can assume that $\xi'$ is $K$-bounded with $\|\xi'_i - \xi\| \leq 3/n$.  If $i \in U_n$ for all $i$, then $\xi_i$ is itself central and so we can set $\xi'_i = \xi_i$.  If $i\notin U_1$, then set $\xi_i'\in H_i$ to be an arbitrary $K$-bounded vector.    From this we see that $\xi$ can be represented by a sequence of central $K$-bounded vectors which concludes the proof.
\end{proof}

\begin{question}
Is the previous theorem true for an arbitrary (separable) tracial von Neumann algebra?
\end{question}

We end this section with some speculation about the model-theoretic meaning of relative property (T).  Suppose that, as above, $M$ is a tracial von Neumann algebra and $B$ is a von Neumann subalgebra.  We say that $B$ has relative property (T) in $M$ if, in Definition \ref{Tdef}, we only conclude the existence of a $B$-central vector $\eta$, that is, a vector $\eta\in H$ such that $[b,\eta]=0$ for all $b\in B$.  The prototypical example of an algebra without property (T) which has relative property (T) in some larger algebra is $L(\mathbb{Z}^n)$ inside of $L(\mathbb{Z}\rtimes \operatorname{SL}_n(\mathbb{Z}))$, $n\geq 2$.

It is tempting to guess that $B$ has relative property (T) inside of $M$ if and only the set of $B$-central vectors is a definable set relative to the theory of $M$-$M$ correspondences.  However, this is not the case.  Indeed, if the set of $B$-central vectors were a definable set relative to the theory of $M$-$M$ correspondences, then the finite sets appearing in the definition of relative property (T) could be chosen to be subsets of $B$ itself, which is far from the case in general.  (In particular, this would imply that $B$ itself has property (T).)

Nevertheless, by \cite[Section 14.5]{Popa-book}, there does appear to be some model-theoretic meaning to relative property (T).  To explain this, we adopt some notation.
Fix a finite set $F \subset M$, $\delta > 0$ and $K \in \bbN$. Let the theory $T(F,\delta,K)$ be the theory of $M$-$M$ correspondences together with the sentence expressing the statement that there is a $K$-bounded unit vector $\xi$ such that 
\[
\|[x,\xi]\| \leq \delta, |\ip{x\xi}{\xi} - \tau(x)| \leq \delta \text{, and }  |\ip{\xi x}{\xi} - \tau(x)| \leq \delta
\]
for all $x \in F$. In the aforementioned reference, it is shown that for a II$_1$ factor $M$ and a subalgebra $B$, $B$ has relative property (T) in $M$ if for some $F$, $\delta$ and $K$, the theory $T(F,\delta,K)$ cannot omit the type of a $B$-central vector.

\section{$\sigma$-finite von Neumann algebras}

The use of bounded vectors gives a new approach to capturing model theoretically the class of $\sigma$-finite von Neumann algebras.  This class is already known to be elementary via different techniques; see \cite{Dab}.  The approach here is different but both presentations rely on the Ocneanu ultraproduct, described in the next section.

\subsection{Preliminaries concerning $\varphi$-right bounded elements}  The class $\svna$ of structures we wish to capture is the class of all pairs $(M,\varphi)$, where $M$ is a von Neumann algebra and $\varphi$ is a faithful normal state on $M$.  Any such $M$ is automatically $\sigma$-finite.  Unlike the case of finite factors, the choice of faithful normal state is not canonical.  However, the GNS constructions induced by any two such states are unitarily conjugate and give rise to a canonical ``standard form.'' (See \cite[Chapter IX, section 1]{Takesaki} for a detailed description.)

We first recall the GNS construction for faithful normal states.  Suppose that $(M,\varphi)\in \svna$.  Then one introduces the inner product $\ip{\cdot}{\cdot}_\varphi$ on $M$ given by
\[
\ip{a}{b}_\varphi = \varphi(b^*a)
\]
with corresponding norm
\[
\|a\|_\varphi = \sqrt{\ip{a}{a}_\varphi}.
\] 
We write $H_\varphi$ for the Hilbert space completion of $(M,\ip{\cdot}{\cdot}_\varphi)$.

Notice that left multiplication gives a faithful *-representation of $M$ on $H_\varphi$ and that the topology induced by $\|\cdot\|_\varphi$ agrees with the strong topology on operator norm-bounded sets of $M$.  However, in general, multiplication on the right does not lead to a bounded operator on $H_\varphi$.  Thus, we say that $a \in M$ is \textbf{$\varphi$-right $K$-bounded} if, for all $b \in M$, we have
\[
\ip{ba}{ba}_\varphi \leq K\ip{b}{b}_\varphi.
\]
If $a$ is $\varphi$-right $K$-bounded, then right multiplication on $H_\varphi$ is a bounded operator with norm at most $\sqrt{K}$.  We say that $a\in M$ is \textbf{$\varphi$-right bounded} if it is $\varphi$-right $K$-bounded for some $K$.

The following fact is probably well-known to experts, but we could not find a precise formulation of it in the literature:

\begin{prop}\label{phidense}
Fix a faithful normal state $\varphi$ on $M$.  Then the set of $\varphi$-right bounded elements of $M$ is strongly dense in $M$.
\end{prop}  

Before proving Proposition \ref{phidense}, we need a lemma:

\begin{lem} Let $M$ be a von Neumann algebra and $\vp, \psi\in M_\ast$ be two faithful states. Then there is an increasing sequence $p_K$ of projections strongly converging to $1$ so that $\vp(p_Kxp_K)\leq K\psi(p_K x p_K)$ for all $x\in M_+$. 

\end{lem}

\begin{proof} Set $\theta_K := K\psi - \vp$. Each $\theta_K$ is a hermitian, normal linear functional, whence \cite[Theorem III.4.2]{Takesaki} implies that there is a projection $p_K$ such that $\theta_{K,+} := \theta_K(p_K \cdot p_K)$ and $\theta_{K,-} := -\theta(p_K^\perp \cdot p_K^\perp)$ are positive and $\theta_K = \theta_{K,+} - \theta_{K,-}$. For $K\geq 1$, $p_K\not=0$ and the sequence $(p_K)$ is increasing in $K$. Set $p := \bigvee_K p_K$.  Assume, towards a contradiction, that $p\not= 1$. Then we have that $\theta_K(p^\perp) = \theta_{K,-}(p^\perp)\leq 0$ whence $\vp(p^\perp)\geq K\psi(p^\perp)$ for all $K$, which is absurd since $\psi$ is faithful.
\end{proof}

\begin{proof}[Proof of Proposition \ref{phidense}]
Fix a unitary $u\in M$; it suffices to show that $u$ is the strong limit of a sequence of $\varphi$-bounded elements of $M$.  By the previous lemma, there is a sequence $p_K$ of projections converging strongly to $1$ so that, for all $x\in M_+$, we have $\varphi(u^*p_Kxp_Ku)\leq K\varphi(p_Kxp_K)$.  Since the sequence $p_Ku$ converges strongly to $u$, it suffices to check that each $p_Ku$ is $\varphi$-right $K$-bounded.  Towards that end, fix $x\in M$ and observe that
$$\|xp_Ku\|_\varphi^2=\varphi(u^*p_Kx^*xp_Ku)\leq K\varphi(p_Kx^*xp_K)\leq K\|x\|_\varphi^2.$$
\end{proof}

We also introduce the norm $\|\cdot\|_\varphi^\#$ given by
$$\|a\|_\varphi^\#:=\sqrt{\varphi(a^*a+aa^*)}.$$
We recall the following facts about $\|\cdot\|_\varphi^\#$:
\begin{fact}\label{hashtagfacts}
Suppose that $(M,\varphi)\in \svna$.
\begin{itemize}
\item $\|\cdot\|_\varphi^\#$ introduces the strong*-topology on operator norm-bounded subsets of $M$.
\item The adjoint is an isometry with respect to $\|\cdot\|_\varphi^\#$.
\item Operator norm-bounded subsets of $M$ are strong*-closed.
\item For any $K$, the strong and strong* topologies agree on operator norm-bounded $\varphi$-right $K$-bounded subsets of $M$.
\end{itemize}
\end{fact}

\subsection{Elementarity}  Motivated by the discussion in the previous subsection, we now introduce a language $\cL_\sigma$:
\begin{enumerate}
\item For each $K$ and $N$, there is a sort $S_{K,N}$ meant to capture the $\varphi$-right $K$-bounded elements of operator norm at most $N$. The metric on these sorts is induced by $\|\cdot\|_\varphi^\#$.  As usual, we have embeddings between the sorts.
\item Addition is, as usual, divided up sort by sort; multiplication sends $S_{K,N} \times S_{K,N}$ to $S_{K^2,N^2}$, and is linear, with operator norm $N$, in each variable.  
\item There are function symbols for the adjoint acting on each sort.  The adjoint is 1-Lipschitz.
\item $\varphi$ is a relation with the obvious range and continuity moduli on each sort.
\end{enumerate}

We associate to each $(M,\varphi)$ in $\svna$ an $\cL_\sigma$-structure $\M(M,\varphi)$, once again called its \textbf{dissection}, by interpreting the sort $S_{K,N}$ as the set of $\varphi$-right $K$-bounded vectors of operator norm at most $N$.  Notice that this does indeed yield an $\cL_\sigma$ structure, the most subtle point being the completeness of each sort.  To see this, suppose that $(a_n)$ is a Cauchy sequence in $S_{K,N}(M,\varphi)$.  By Fact \ref{hashtagfacts}, $a_n$ strong* converges to an element $a\in M$ of operator norm at most $N$; it is easy to verify that this $a$ is also $\varphi$-right $K$-bounded.

Let $\cal S$ be the class of all dissections of elements of $\svna$. 
\begin{thm}
$\cal S$ is an elementary class which is categorically equivalent to $\svna$.
\end{thm}

\begin{proof} We first check closure under ultraproducts.  Suppose that $(M_i,\varphi_i) \in \cal S$ for $i \in I$ and $\cU$ is an ultrafilter on $I$. Form the ultraproduct of the $\cL$-structures $\M(M_i,\varphi_i)$ and let $M$ be the union of the sorts $\prod_\cU S_{K,N}(M_i)$.  Note that $M$ is clearly a $*$-algebra.  Set $\varphi:=\lim_\cU \varphi_i$;  note that $\varphi$ is a faithful normal state on $M$. Let $H_\varphi$ be the GNS Hilbert space associated to $(M,\varphi)$ and let $\tilde{M}$ be the strong closure of $M$ in this representation. We have that $(M, H_\vp)$ is a left Hilbert algebra in the sense of  \cite[section VI.1]{Takesaki}. For $x\in M$, let $\pi_r(x)$ be the operator of right multiplication of $x$ on $H_\vp$. Note that $\|\pi_r(x)\|\leq \sqrt K$ if and only if $x$ is $\vp$-right bounded with constant $K$. Again by \cite[Lemma VI.1.8]{Takesaki} we have that $\pi_r(M)'' = \tilde M'$, whence $\pi_r(M)$ is strongly dense in $\tilde M'$ by the double commutant theorem.

 It suffices to show that $S_{K,N}(\tilde{M},\varphi)\subseteq \prod_\cU S_{K,N}(M_i)$.  Towards this end, suppose that $a\in S_{K,N}(\tilde{M},\varphi)$, whence $\pi_r(a)$ is defined and $\|\pi_r(a)\|\leq \sqrt K$.  Consider the von Neumann algebra $\cal D := \{x\oplus \pi_r(x) : x\in M\}''\subset \cal B(H_\vp\oplus H_\vp)$ Since $a\oplus \pi_r(a)\in \cal D$, by the Kaplansky density theorem and Fact \ref{hashtagfacts} there is a sequence $a_j\in M$ of elements such that $\|a_j\oplus \pi_r(a_j)\|\leq \max\{N,\sqrt K\}$ and $a_j$ converges to $a$ in the strong* topology. A standard functional calculus argument shows that we can then take $a\in \prod_\cU S_{K,N}(M_i)$, as desired.


We now check closure under ultraroots.  Suppose that $\M$ is a $\cL_\sigma$-structure and $\cU$ is an ultrafilter such that $\M^\cU$ is the dissection of $(M,\varphi)\in \svna$.  We let $N$ be the von Neumann subalgebra of $M$ generated by $\bigcup_{K,N}S_{K,N}(\M)$.  It suffices to show that $\M$ is the dissection of $(N,\varphi)$.  Once again, the only subtle point is in showing that if $a\in N$ has operator norm at most $N$ and is $\varphi$-right $K$-bounded, then $a\in S_{K,N}(\M)$.  Since $a$ is $\varphi$-right $K$-bounded and $\M$ is an elementary substructure of $\M^\cU$, we have that $\|xa\|_\varphi^\#\leq \sqrt{K}\|x\|_\varphi^\#$ for all $\varphi$-right bounded elements $x$ of $M$.  Since the $\varphi$-bounded elements of $M$ are strongly dense in $M$, it follows that $a$ is also a $\varphi$-right $K$-bounded element of $M$, whence $a\in S_{K,N}(\M^\cU)$.  As before, another application of elementarity implies that $a\in S_{K,N}(\M)$, as desired.  
\end{proof}  

Let $T_\sigma$ be a theory axiomatizing $\cal S$.  As in the case of correspondences, the proof of closure under ultraroot only used the weaker assumption that an ultraroot is existentially closed in its ultrapower.  Thus, we have:

\begin{cor}
$T_\sigma$ is $\forall\exists$-axiomatizable.
\end{cor}

\subsection{Connection with the Ocneanu ultraproduct}  Suppose that, for $i\in I$, we have a member $(M_i,\varphi_i)$ of $\svna$.  Moreover, assume that $\cU$ is an ultrafilter of $I$.  Set
\[
\cal I:=\{(m_i) \in \ell^\infty(M_i,I) : \lim_\cU \|m_i\|_{\varphi_i}^\# = 0 \}
\]
 and set
 \[
M := \{ (m_i) \in \ell^\infty(M_i,I) : (m_i)\cal I+\cal I(m_i)\subset \cal I \}.
\] 
By construction, $\cal I$ is a two-sided ideal of $M$.  The \textbf{Ocneanu ultraproduct} is the quotient $\prod^\cU M_i := M/\cal I$.  This is a $\sigma$-finite von Neumann algebra with faithful normal state given by $\lim_\cU \varphi_i$.  Given $(m_i)\in M$, write $(m_i)^\star$ for its equivalence class in $\prod^\cU M_i$.


We now show that the Ocneanu ultraproduct and the ultraproduct introduced here are the same.

\begin{prop}
Suppose that $(M_i,\varphi_i)\in \svna$ for all $i \in I$ and $\cU$ is an ultrafilter on $I$. Then, for all $K$ and $N$, we have:
\[
\textstyle{\prod_{\cU}S_{K,N}(M_i,\varphi_i)=S_{K,N}(\prod^{\cU}(M_i,\varphi_i)).}
\]
In other words, the dissection of the Ocneanu ultraproduct is the ultraproduct of the dissections, whence $\prod_\cU M_i = \prod^\cU M_i$.
\end{prop}

\begin{proof} 


First suppose that $(a_i)^\bullet \in \prod_\cU S_K(M_i)$.  It is clear that $(a_i) \in \ell^\infty(M_i,I)$.  Suppose that $(m_i) \in J$; we show that $(a_im_i),(m_ia_i) \in J$.  Suppose first that $(a_im_i) \notin J$.  Let $\lim_\cU \|a_im_i\|_{\varphi_i}^\# = L \neq 0$.  Consider all $i \in I$ such that $\|m_i\|_{\varphi_i}^\# < L/2N$.  For each such $i$, $m_i$ demonstrates that the operator norm of $a_i$ is greater than $N$, which is a contradiction.  The proof that $(m_ia_i)\in J$ proceeds similarly, using that each $a_i$ is $\varphi_i$-right $K$-bounded.  It is now clear that $(a_i)^\bullet=(a_i)^\star$ and that this element is in $S_{K,N}(\prod^{\cU}(M_i,\varphi_i))$.

A functional calculus argument similar to that showing that $\cal S$ is closed under ultraproducts can be used to show that any $a\in S_{K,N}(\prod^\cU M_i)$ can be represented as $a=(a_i)^\star$ with each $a_i\in S_{K,N}(M_i)$, establishing the other inclusion.

\end{proof}

\subsection{Definability of the modular automorphism group}

We now explain how to capture the modular automorphism group in the language of $\sigma$-finite von Neumann algebras.  Given $(M,\varphi)$ where $\varphi$ is a normal faithful state on $M$, let $\Delta_\varphi$ be the modular operator with respect to $\varphi$.  For any $t \in \bbR$, let $\sigma_t^\varphi(x) = \Delta_\varphi^{it} x \Delta_\varphi^{-it}$.  $\sigma_t^\varphi$ is an automorphism of $M$ and the map $t\mapsto \sigma_t$ is a 1-parameter subgroup of $\operatorname{Aut}(M)$, called the modular automorphism group of $\varphi$.  The following is Theorem 4.1 from \cite{AH} with the only modification being that it is stated for arbitrary ultraproducts.

\begin{thm}
Suppose that $(M_i,\varphi_i)$ are $\sigma$-finite von Neumann algebras with faithful normal states $\varphi_i$ for every $i \in I$.  Moreover, suppose that $\cU$ is an ultrafilter on $I$ and $t \in \bbR$.  Let $(M,\varphi) = \prod_\cU (M_i,\varphi_i)$.  Then, for any $(x_i)^\bullet\in M$, we have
\[
\sigma_t^\varphi(x_i)^\bullet = (\sigma_t^{\varphi_i}(x_i))^\bullet.
\]
\end{thm}

An immediate consequence of this, by Beth's definability theorem, is
\begin{cor}
For all $t\in \bbR$, $\sigma_t$ is a $T_\sigma$-definable function.  Moreover, if $\varphi_t$ is a $T_\sigma$-definable predicate defining $\sigma_t$, then the map $t\mapsto \varphi_t$ is continuous with respect to the logic topology.
\end{cor}


%
%
%
%


\subsection{Axiomatizable and local classes}

In this subsection, we determine whether or not natural subclasses of $\cal C$ are axiomatizable or local (defined below).  A similar endeavour for the class of tracial von Neumann algebras was undertaken in \cite{FHS3}.

The following is \cite[Proposition 6.3]{AH}:
\begin{fact}
There is a family $\varphi_n$ of faithful normal states on the hyperfinite II$_1$ factor $\mathcal{R}$ such that $\prod_{\cU}(\mathcal{R},\varphi_n)$ is \emph{not} semifinite.
\end{fact}

In what follows, given a property $P$ of von Neumann algebras, we abuse terminology and speak of the class of algebras satisfying $P$ when referring to the class of pairs $(M,\varphi)\in \svna$ where $M$ satisfies property $P$.  We also say that the class of algebras satisfying P is axiomatizable if the class of dissections of pairs $(M,\varphi)$ with $M$ satisfying P is an axiomatizable class.

\begin{cor}
The following classes are not elementary classes:  finite algebras, semifinite algebras, II$_1$ factors.
\end{cor}

\begin{prop}
The class of II$_\infty$ factors is also not axiomatizable. 
\end{prop}

\begin{proof}
Let $M$ be a II$_1$ factor and set $N:=M\otimes \mathcal{B}(\ell^2)$.   Let $(e_i)$ be the standard orthonormal basis for $\ell^2$. Choose states $\varphi_n(x\otimes T):=\tau(x)\left(\sum_i \lambda_{n,i} \ip{Te_i}{ e_i}\right)$ with $\lambda_{n,i} \not= 0$, whence they are faithful. Further assume that $\lambda_{n,1}\to 1$ while the rest tend to zero. Then $\prod_{\cU}(N,\varphi_n)\cong (M,\tau_M)^{\cU}$, a II$_1$ factor.
\end{proof}

As pointed out at the end of Section 3.1 of \cite{AH}, given a von Neumann algebra $\varphi$ and two faithful normal states $\varphi_1$ and $\varphi_2$, we have $\prod_{\cU}(M,\varphi_1)\cong \prod_{\cU}(M,\varphi_2)$, whence there is a well-defined notion of the \emph{ultrapower} of $M$.  In this case, we denote the algebra simply by $M^{\cU}$.  The following is \cite[Theorem 6.18]{AH}:

\begin{fact}\label{notfactor}
Suppose that $M$ is a $\sigma$-finite factor of type III$_0$.  Then $M^\cU$ is \emph{not} a factor.
\end{fact}

Recall that a class of structures is called \textbf{local} if it is closed under elementary equivalence; equivalently, the class is closed under isomorphism, ultraroot, and ultrapower.

\begin{cor}
The following classes are not local:  factors, type III$_0$ factors.
\end{cor}

However:

\begin{prop}
The following classes are local:  finite algebras, semi-finite algebras, type I$_n$ factors (for a fixed $n$), type II$_1$ factors, type II$_\infty$ factors.
\end{prop}

\begin{proof}
The only classes that need explanation are that of the case of semi-finite algebras and type II$_\infty$ factors; these classes are seen to be local by using the fact (which appears to be folklore) that, given a von Neumann algebra $M$ with separable predual, a separable Hilbert space $H$, and an ultrafilter $\cU$ on $\mathbb{N}$, that $(M\overline{\otimes} \cal B(H))^{\cU}\cong M^{\cU}\overline{\otimes} \cal B(H)$.
\end{proof}

Finally, we mention:

\begin{prop}
Suppose that $S\subseteq (0,1]$ is closed.  Then the set of type III$_\lambda$ factors where $\lambda\in S$ is an axiomatizable class.
\end{prop}

\begin{proof}
In \cite[Theorem 6.11]{AH}, it is shown that if $M$ is a type III$_\lambda$ factor for $\lambda\in (0,1]$ and $\varphi_n$ is any family of faithful normal states on $M$, then $\prod_{\cU}(M,\varphi_n)$ is once again a type III$_\lambda$ factor.  The same proof shows that the class of pairs $(M,\varphi)$, where $M$ is a type III$_\lambda$ factor with $\lambda\in S$, is closed under ultraproducts.  

It remains to see that this class is closed under ultraroots.  Thus, suppose that $(M,\varphi)$ is such that $M^\cU$ is a type III$_\lambda$ factor with $\lambda\in S$.  It is easy to see that $M$ must then be a factor.  $M$ cannot be semifinite for then $M^\cU$ would also be semifinite.  $M$ cannot be III$_0$ for then $M^\cU$ would not be a factor by Fact \ref{notfactor}.  Thus, $M$ must be type III$_\eta$ for $\eta\in (0,1]$; by the previous paragraph, it must be that $\eta=\lambda$.
\end{proof}

\section{A family of Connes-type ultraproducts on \cstar-algebras}

A \textbf{statial \cstar-algebra} is a pair $(A,\Phi)$, where $A$ is a unital \cstar-algebra and $\Phi$ is a subset of the state space $\cal S(A)$ of $A$.  The purpose of this section is to introduce an ultraproduct construction for a natural class of statial \cstar-algebras.
\subsection{Bounded vectors in statial \cstar-algebras}

%


Until further notice, fix a statial \cstar-algebra $(A,\Phi)$. For $x\in A$, we define the semi-norm $\|\cdot\|_{2,\Phi}$ by $$\|x\|_{2,\Phi} := \sup_{\phi\in\Phi} \max\{\phi(x^*x), \phi(xx^*)\}^{1/2}.$$ We say that the statial \cstar-algebra $(A,\Phi)$ is \textbf{faithful} if $\|\cdot\|_{2,\Phi}$ is a norm on $A$. Suppose, from now on, that $(A,\Phi)$ is also faithful.  It follows, for $x\in A$, that $\|x\|_{2,\Phi}\leq \|x\|$; moreover, equality holds for all $x\in A$ if and only if the closed, convex hull of $\Phi$ equals $\cal S(A)$. 

We let $L^2(A,\Phi)$ denote the completion of $A$ with respect to the norm $\|\cdot\|_{2,\Phi}$. Note that this Banach space is not necessarily a Hilbert space, although we keep the notation to emphasize the similarity with the single trace case. 
Finally note that the involution on $A$ induces an antilinear isometry $J_\Phi$ of $L^2(A,\Phi)$.

We say that a vector $\xi\in L^2(A,\Phi)$ is \textbf{$K$-bounded} if, for all $a\in A$, we have $$\max\{\|a\xi\|_{2,\Phi}, \|\xi a\|_{2,\Phi}\}\leq K \|a\|_{2,\Phi}.$$ In other words, $\xi$ is $K$-bounded if there are bounded operators $\pi_\ell(\xi), \pi_r(\xi)\in \cal B(L^2(A,\Phi))$ with $\|\pi_\ell(\xi)\|, \|\pi_r(\xi)\|\leq K$ for which $$\pi_\ell(\xi)a = \xi a\textup{, and}\ \pi_r(\xi)a = a\xi$$ for all $a\in A$. We set $\cal A_\Phi$ to be the set of all bounded vectors in $L^2(A,\Phi)$. There is a natural involution $\star$ on $\cal A_\Phi$ given by $\xi^\star := J_\Phi\xi$.  It follows that $\pi_\ell(\xi^\star) = J_\Phi \pi_r(\xi) J_\Phi$ and $\pi_r(\xi^\star) = J_\Phi \pi_\ell(\xi) J_\Phi$. This endows $\cal A_\Phi$ with the structure of an \textbf{involutive Banach algebra}, that is, a unital Banach algebra equipped with an isometric involution. It is not clear, in general, whether $\cal A_\Phi$ has the structure of a C$^*$-algebra.  However, under a natural assumption to be described below, we can show that $\cal A_\Phi$ always admits a compatible C$^*$-algebraic structure.

We say that the statial \cstar-algebra $(A,\Phi)$ is {\bf full} if $\Phi$ is  invariant by unitary conjugation.

\begin{prop}
Suppose that $(A,\Phi)$ is a full and faithful statial \cstar-algebra.  Then, for any $x\in A$, we have $$\|x\|_{2,\Phi} = \sup_{\vp\in\Phi} \vp(x^*x)^{1/2} = \sup_{\vp\in \Phi} \vp(xx^*)^{1/2}.$$ In particular, every $a\in A$ is $\|a\|$-bounded.
\end{prop}

\begin{proof}
If $x$ is invertible, the identity follows easily by polar decomposition for invertible elements in C$^*$-algebras. Let $x_\la = \la1 + x$, which for $\ab{\la}> \|x\|$ is invertible. Choose $\ab{\la} = \|x\|+1$, say, for which ${\rm Re}(\la x) =0$. For $\vp\in \Phi$ we have $\psi\in\Phi$ such that $\vp(x_{\la}^*x_{\la}) = \psi(x_\la x_\la^*)$ from which it easily follows that $\vp(x^*x) = \psi(xx^*)$.
\end{proof}


We are now ready to prove that the set of bounded elements associated to a full and faithful statial \cstar-algebra admits the structure of a \cstar-algebra.  We first recall the fact that every involutive Banach algebra $\cal A$ admits an essentially unique contractive $\ast$-homomorphism with dense image into a C$^*$-algebra, denoted $\cal A^{\text{\cstar}}$, such that every contractive $\ast$-homomorphism from $\cal A$ into a C$^*$-algebra factors through $\cal A^{C^\ast}$  \cite[section 2.7]{Dix}.  By the previous proposition, given a full and faithful statial \cstar-algebra $(A,\Phi)$, there is a natural inclusion $A\subseteq (\cal A_\Phi)^{\text{\cstar}}$.    
\begin{prop}\label{fullc*} Suppose that $(A,\Phi)$ is a full and faithful statial \cstar-algebra.  Then $\cal A_\Phi$ admits an equivalent C$^*$-algebra norm.
\end{prop}

\begin{proof}
Every $\vp\in\Phi$ extends continuously to a linear functional $\tilde\vp$ on $\cal A_\Phi$. By Cauchy-Schwartz, we have, for all $a\in A\setminus\{0\}$, that $$x\mapsto \frac{\tilde\vp(a^*xa)}{\|a\|_{2,\Phi}^2}:\cal A_\Phi\to \mathbb{C}$$ is a contractive, positive linear functional of $\cal A_\Phi$. Since the supremum over all of these functionals computes norm squared of $x^*x$ for all $x\in\cal A_\Phi$, we have by \cite[2.7.2]{Dix} that $\|x^*x\|$ agrees with the norm of $x^*x$ as an element of $(\cal A_\Phi)^{\text{\cstar}}$, whence there is an isometric identification of the positive elements and $\cal A_\Phi$ is complete in the norm on $(\cal A_\Phi)^{\text{\cstar}}$.
\end{proof}

We end this section with a couple of further observations:

\begin{prop}
Suppose that $(A,\Phi)$ is a full and faithful statial \cstar-algebra.  We then have:
\begin{enumerate}
\item Any closed, bounded subset of $\cal A_\Phi$ is complete in the $(2,\Phi)$-norm.
\item $\pi_\ell(\cal A_\Phi)$ is the commutant in $\cal B(L^2(A,\Phi))$ of $\pi_r(A)$.
\end{enumerate}
\end{prop}

\begin{proof}
Statement (1) is a standard application of the triangle inequality.  For (2), let $\hat a$ denote an element $a\in A$ considered as vector in $L^2(A,\Phi)$. For $T\in \pi_r(A)' \cap \cal B(L^2(A,\Phi)$, setting $\xi = T\hat 1$, we have that $\xi a = \pi_r(a)T\hat 1 = T\pi_r(a)\hat 1 = T(\hat a)$ whence $T = \pi_\ell(\xi)$. \qedhere

\end{proof}



\subsection{Ultraproducts of statial \cstar-algebras} Let $(A_i,\Phi_i)_{i\in I}$ be a sequence of full and faithful statial \cstar-algebras. For $\cU$ a non-principal ultrafilter on $I$, we let $\prod_{\cU} L^2(A_i,\Phi_i)$ denote the Banach space ultraproduct of $L^2(A_i,\Phi_i)_{i\in I}$. We define the \textbf{Connes ultraproduct} of the sequence $(A_i,\Phi_i)$, denoted, $\prod_{C\cU} (A_i,\Phi_i)$,  to be the subset of $\prod_{\cU} L^2(A_i, \Phi_i)$ consisting of all \textbf{uniformly bounded} elements, that is, elements $\xi$ for which there are $K\in \bbN$ and $\xi_i\in (\cal A_i)_{\Phi_i}$ which are all $K$-bounded and for which $\xi=(\xi_i)^\bullet$.

\begin{prop}
The Connes ultraproduct $\prod_{C\cU} (A_i,\Phi_i)$ is a \cstar-algebra.
\end{prop}

\begin{proof} 
The only thing to note is that the proof of Proposition \ref{fullc*} actually shows that the Banach and $C^*$-norms on $\cal A_{\Phi}$ are each dominated by $4$ times the other.
\end{proof}

As in section 3, there is a natural formulation of the Connes ultraproduct in terms of points of continuity.  We say that $\xi = (\xi_i)^\bullet\in \prod_{\cU} L^2(A_i, \Phi_i)$ is a point of continuity if: for all $a = (a_i)^\bullet\in \prod_{C\cU}(A_i,\Phi_i)$, we have that $(a_i)^\bullet\mapsto (\pi_\ell(a_i)\xi_i)^\bullet$ and $(a_i)^\bullet\mapsto (\pi_r(a_i)\xi_i)^\bullet$ are independent of choices of representatives for $a$ and $\xi$ and the consequently well-defined maps extend continuously to $\prod_{\cU} L^2(A_i,\Phi_i)$. 

A similar analysis as in Section 3 above shows that:

\begin{prop} For any sequence of full and faithful statial \cstar-algebras $(A_i,\Phi_i)_{i\in I}$, we have that $\prod_{C\cU} (A_i,\Phi_i)$ is the closure of the set of points of continuity.
\end{prop}

If $\Phi_i = \cal S(A_i)$ for all $i$, then it is clear that $\prod_{C\cU} L^2(A_i,\Phi_i)$ is simply identifiable with the \cstar-algebraic ultraproduct $\prod_{\cU} A_i$. More generally, we have:

\begin{prop} For any sequence $(A_i,\Phi_i)_{i\in I}$ of full and faithful statial \cstar-algebras, we have that $\prod_{C\cU} (A_i,\Phi_i)$ is a quotient of $\prod_{\cU} (\cal A_i)_{\Phi_i}$.
\end{prop}

Let $\fr C$ be an elementary class of unital \cstar-algebras. Let $\cal F$ be a contravariant functor that assigns to each $A\in \fr C$ a full and faithful subset $\cal F(A)$ of $\cal S(A)$.

\begin{prop} \label{prop-uniform2norm} For any sequence $(A_i)_{i\in I}$ from $\fr C$, setting $D := \prod_{\cU} A_i$, we have that $\prod_{C\cU} (A_i, \cal F(A_i))\cong \cal D_{\cal F(D)}$ if and only if the $(2,\cal F(D))$-norm is expressible as  $\|(x_i)^\bullet\|_{2,\cal F(D)} =\lim_\cU \|x_i\|_{2, \cal F(A_i)}$ for all $(x_i)^\bullet\in \prod_{\cU} A_i$.
\end{prop}

Note that it may be the case that the $(2,\cal F(D))$-norm and $(2,\cal F(A_i))$-norms behave well under ultralimits while the actual functor $\cal F$ does not.

\begin{example}
Suppose that $\fr C$ is the class of $\cal Z$-stable, stably finite C$^*$-algebras admitting at least one tracial state (this is elementary by \cite[Section 2]{munster}) and $\cal F$ is the contravariant functor associating to each element of $\fr C$ its class of tracial states.  Then, by \cite[Theorem 6]{Oz}, if each $A_i$ is exact, we have that each tracial state on $D$ is weak* approximated by an ultralimit of tracial states on the $A_i$s, whence the isomorphism in Proposition \ref{prop-uniform2norm} holds. 
\end{example} 


\appendix

\section{Correspondences and Tensor Products}\label{tensor} Let $M$ and $N$ be von Neumann algebras. A C$^*$-algebra tensor product $M\otimes_\pi N^{op}$ of $M$ and $N^{op}$ is said to be {\bf binormal} if $\pi: M\odot N^{op}\to \cal B(H_\pi)$ is normal when restricted to $M\otimes 1$ and $1\otimes N^{op}$. For any such binormal tensor product, $H_\pi$ naturally has the structure of an $M$-$N$ correspondence, and conversely any correspondence gives rise to a binormal tensor product in the natural way. Thus by summing over all classes of $M$-$N$ correspondences, we see there is a maximal binormal tensor product, $M \otimes_{bin} N^{op}$, in the sense that the identity on $M\odot N^{op}$ extends to a $\ast$-homomorphism $M\otimes_{bin} N^{op}\to M\otimes_\pi N^{op}$ for any binormal tensor $\pi$.

The next result shows that the correspondence Fell topology is literally the restriction of the Fell topology for the representation theory of the pre-C$^*$-algebra $M\odot N^{op}$ to the class of binormal representations.

\begin{prop} Let $(H,\pi)$ and $(K, \rho)$ be $M$-$N$ correspondences. We have that $(K,\rho)$ is weakly contained in $(H,\pi)$ if and only if the identity on $M\odot N^{op}$ extends to a $\ast$-homomorphism $M\otimes_{\pi} N^{op}\to M\otimes_\rho N^{op}$.
\end{prop}

Recall that a tracial von Neumann algebra $(M,\tau)$ is said to be {\bf semidiscrete} if the identity on $M\odot M^{op}$ extends to a $\ast$-homomorphism $M\otimes M^{op}\to C^*(M, JMJ)\subset \cal B(L^2(M,\tau))$.  By the preceding proposition, this is equivalent to the trivial correspondence being weakly contained in the coarse correspondence. It is also well-known that semi-discreteness is equivalent to amenability, see \cite{Takesaki}.

The following proposition is due to Effros and Lance, \cite[Theorem 4.1]{EL}.

\begin{prop}\label{effroslance} We have that $M$ is semidiscrete if and only if the identity extends to a $\ast$-isomorphism $M\otimes_{bin} M^{op}\cong M\otimes M^{op}$, i.e, every $M$-$M$ correspondence is weakly contained in the coarse correspondence. 
\end{prop}


%
%

%
%
%


\end{document}